\documentclass[11pt,reqno]{amsart}
\usepackage{amsmath}
\usepackage{paralist}
\usepackage[misc]{ifsym}
\usepackage{epsfig} 
\usepackage{epstopdf} 
\usepackage[colorlinks=true]{hyperref}
\usepackage{amssymb}
\hypersetup{urlcolor=blue, citecolor=red}
\allowdisplaybreaks

\newtheorem{theorem}{Theorem}[section]
\newtheorem{corollary}[theorem]{Corollary}

\newtheorem{lemma}[theorem]{Lemma}
\newtheorem{proposition}[theorem]{Proposition}

\theoremstyle{definition}

\newtheorem{remark}[theorem]{Remark}

\def\R{{\mathbb{R}}}
\DeclareMathOperator{\interior}{int}

\title[Traveling fronts for a free interface model of autoignition]
{Existence and uniqueness of traveling fronts for a free interface model of autoignition in reactive jets} 

\author[M. Ma, P. V. Gordon, R. Roussarie, P. Shang and C.-M. Brauner]{}

\begin{document}
\vspace*{-2cm}
\maketitle

\centerline{\scshape
Mingxin Ma$^{{\href{mailto:2311739@tongji.edu.cn}{\textrm{\Letter}}}1}$,
 Peter V. Gordon$^{{\href{mailto:gordon@math.kent.edu}{\textrm{\Letter}}}2}$,
 Robert Roussarie$^{{\href{mailto:Robert.Roussarie@u-bourgogne.fr}{\textrm{\Letter}}}3}$,}
 \centerline{\scshape
 Peipei Shang$^{{\href{mailto:shang@tongji.edu.cn}{\textrm{\Letter}}}1}$
 and Claude-Michel Brauner$^{{\href{mailto:claude-michel.brauner@u-bordeaux.fr}{\textrm{\Letter}}}*4}$}

\medskip

{\footnotesize
 \centerline{$^1$School of Mathematical Sciences,}
 \centerline{Key Laboratory of Intelligent Computing and Applications, Ministry of Education,}
  \centerline{Tongji University,  1239 Siping Rd., Shanghai 200092, China}
} 

\medskip

{\footnotesize
 \centerline{$^2$Department of Mathematical Sciences,
  Kent State University,} \centerline{Kent, OH 44242, USA}
}

\medskip

{\footnotesize
 \centerline{$^3$IMB UMR 5584,
  Universit\'e de Bourgogne Europe,} \centerline{CNRS, F-21000 Dijon, France}
}

\medskip

{\footnotesize
 \centerline{$^4$Institut de Math\'ematiques de Bordeaux UMR CNRS 5251, }
 \centerline{Universit\'e de Bordeaux,  33405 Talence, France}
}

\bigskip
%
\begin{center}
	\emph{Dedicated to Roger Temam on his 85th birthday, with respect and admiration}
\end{center}

\begin{abstract}
We consider a one-dimensional reaction-diffusion model with a piecewise continuous reaction term  that  describes the
propagation of autoignition fronts in reactive co-flow jets within a certain parametric regime. The model is reduced to a free boundary
problem with two interfaces. It is shown that this problem admits a permanent traveling front solution which is unique up to translation.
The result is obtained using a dynamical system approach employing  the stable manifold theorem and the Melnikov integral as the main tools.
\end{abstract}
\footnote{*Corresponding author: C.-M. Brauner.}


\section{Introduction}

In this paper, we analyze  a one-dimensional reaction-diffusion model
of  weak reactive  autoignition fronts that were recently  experimentally observed  in co-flow  reactive jets
  in  a certain parametric regime. As the model considered in this
 paper is relatively new, we start with a  brief  discussion of  a framework in which the model was derived.

Autoignition (thermal explosion) is a process of spontaneous development of reaction rates in \textcolor{black}{a reactive medium}. \textcolor{black}{The study of this phenomenon} is one of the fundamental problems
of combustion theory \cite{Sem,FK,ZBLM}. Classical   theory of thermal explosion was developed in pioneering works of  N.N. Semenov and  D.A. Frank-\textcolor{black}{Kamenetskii} \cite{FK,Sem}.
The key assumption of  Semenov-Frank-\textcolor{black}{Kamenetskii}  theory, which is justified by numerous experimental studies,  is that prior to autoignition, event
reaction rates in reactive systems are relatively weak, which results in a negligible consumption of the reactive component. This assumption leads to a substantial \textcolor{black}{simplification} of   equations  governing the dynamics of reactive media in this regime. In particular,   Semenov-Frank-\textcolor{black}{Kamenetskii}  theory allows \textcolor{black}{one} to  reduce  \textcolor{black}{the} problem of autoignition to the analysis
of a single reaction-diffusion equation representing the conservation of energy that governs the dynamics of the temperature field within reactive media.
The general approach of Semenov-Frank-\textcolor{black}{Kamenetskii}  turned out to be very robust and  was used by physicists and engineers   to derive models of  autoignition for
various reactive systems, see, e.g., \cite{Law,Will}.

As a rule, autoignition takes  place in a strongly localized region which is called an  {\it ignition kernel}. It is often the case that the
formation of the ignition kernel results in instantaneous  development of strong chemical reaction through the entire reactive system leading to  thermal explosion.
 This situation is described by \textcolor{black}{the} classical  Semenov-Frank-\textcolor{black}{Kamenetskii} theory. In  the  framework of this theory, the formation of the ignition kernel is associated with the formation of \textcolor{black}{a} singularity (finite time blow-up) in the equation describing the dynamics of the temperature field in reactive media \cite{Brezis}. \textcolor{black}{Moreover, this blow-up is complete, that is, the solution cannot be prolonged beyond the blow-up time in any reasonable way as it becomes infinite everywhere in the domain, see Theorem 1 in \cite{Martel}. }
 The scenario  described above is  typical, yet not universal. Indeed, there are combustion devices
  in which  the formation of an  ignition kernel may \textcolor{black}{be followed} by extinction or \textcolor{black}{by} formation
 of an advancing {\it ignition front}  that propagates through the reactive system.
 Hence, in general, the formation of an ignition kernel is a necessary, but not  a sufficient condition for initiation of the combustion process.
  The presence of an  ignition front was identified in the experimental studies of {\it hydrothermal flames}  (flames that persist in highly compressed and overheated water)
 performed at \textcolor{black}{the} high pressure laboratory of NASA Glenn Research Center \cite{exp1,exp2}. In these studies, a co-flow jet consisting of a fuel and oxidizer was injected into \textcolor{black}{a} combustion
 vessel filled with water at \textcolor{black}{a} supercritical state. It was observed that in this \textcolor{black}{experimental} setting\textcolor{black}{,} under appropriate \textcolor{black}{experimental} conditions\textcolor{black}{,} the formation of an ignition kernel,
  \textcolor{black}{analyzed in detail} in \cite{GMN,GHH,GHHK,GGHHKS}, \textcolor{black}{was}  followed by the development
 of an ignition front that propagated toward the injection inlet.   \textcolor{black}{The principal difference} of an ignition front from the conventional flame front is that the reaction within the
 ignition front remains relatively weak. The  experimental observation reported in  \cite{exp1,exp2}  strongly suggests  that the  main assumption of Semenov-Frank-\textcolor{black}{Kamenetskii} theory remains applicable  for modeling of  ignition fronts.
 Two models \textcolor{black}{for the} propagation of ignition fronts in reactive jets\textcolor{black}{, which} utilize Semenov-Frank-\textcolor{black}{Kamenetskii} and  Burke-Schumann theories  as well as  several other simplifying assumptions\textcolor{black}{,} were proposed in \cite{GHH2,MHHG} to gain some qualitative understanding of the mechanisms driving the ignition fronts in reactive co-flow jets  observed in    \cite{exp1,exp2}.
 While the result of the analysis of \textcolor{black}{the} first model \cite{GHH2,BSZ} provided some nontrivial insight \textcolor{black}{into the} propagation of the ignition fronts, some of the important aspects of their dynamics
 remained unexplored. The more sophisticated second model  \cite{MHHG}, despite  its  relative simplicity,   reproduced  all principal
 experimentally observed parametric  regimes the system exhibits: extinction, blow off,   two distinct modes of propagation of the ignition fronts and thermal explosion.
 Regimes relevant to propagation of the ignition fronts will be discussed later in this section.
 We also note that the problem of parametric extinction of hydrothermal   flames  in co-flow reactive jets was  addressed in \cite{GHH3}.

 The model for propagation of the ignition\, front in co-flow\, reactive jet proposed

 \ \vspace*{-14pt}\hspace*{-0.67cm}
  in \cite{MHHG} reads:

 \begin{equation}\label{eq:i1}
 \left\{
 \begin{array}{lll}
 \Delta \Theta -\tilde c\frac{\partial \Theta}{\partial x}-\tilde h \tilde g(\Theta)=0 & \mbox{in} & \Omega\times \mathbb{R},\\
 \nabla \Theta \cdot {\nu} =\tilde q \tilde f(\Theta) &\mbox{on} & \partial \Omega\times \mathbb{R},\\
 \Theta \to \Theta_+ & \mbox{as} & x\to +\infty,\\
 \Theta \to 0 & \mbox{as} & x\to - \infty,
 \end{array}
 \right.
 \end{equation}
 where $\Omega\subset \mathbb{R}^2$ is an open bounded set representing the cross section of the jet;
 $\partial \Omega$ is the boundary of the jet's cross section, \textcolor{black}{which is assumed to be sufficiently smooth}; $\bf{\nu}$ \textcolor{black}{is the} outward normal to the boundary of the jet;
 $\Theta$ is an appropriately
 normalized temperature within the jet; \textcolor{black}{$x$  is a coordinate along the jet axis
pointing in the direction of the injected flow  (note that this coordinate system is opposite to the one used in \cite{MHHG});} 
 $\tilde c$ is an a priori unknown speed of the ignition front;
  $\tilde g(\Theta)$ is \textcolor{black}{a} non-negative, non-decreasing function describing volumetric heat loss;
 $\tilde f(\Theta)$  is a reaction function of the ignition type
 that represents the  heat production on the jet's boundary due to the chemical reaction on a contact surface of the fuel and oxidizer, namely $$\tilde f(\Theta)=f(\Theta) H(\Theta-\theta_{ig}),$$ where $f$ is \textcolor{black}{a} non-negative, non-decreasing
 function, $H$ is the Heaviside step function and  $\theta_{ig}$ is the ignition temperature;
  \textcolor{black}{$\tilde h$ and $\tilde q$} represent the intensities of the heat loss and chemical reaction respectively; and
 $\Theta_+$ is the temperature distribution far behind the ignition front given by the largest positive stable solution (provided it exists)  of the following boundary value problem:
\begin{eqnarray}\label{eq:i2}
\left\{
 \begin{array}{lll}
 \Delta \Theta_+= \tilde h \tilde g(\Theta_+) & \mbox{in} & \Omega, \\
 \nabla \Theta_+ \cdot {\nu} =\tilde q \tilde f(\Theta_+)  & \mbox{on} & \partial \Omega.
 \end{array}
 \right.
 \end{eqnarray}
 \textcolor{black}{A sketch of the reactive co-flow jet configuration used in the derivation of model \eqref{eq:i1} is depicted in Figure \ref{fig-add}.}
 \begin{figure}[h]\vspace*{-14pt}
\centering \includegraphics[width=3in, height=3in]{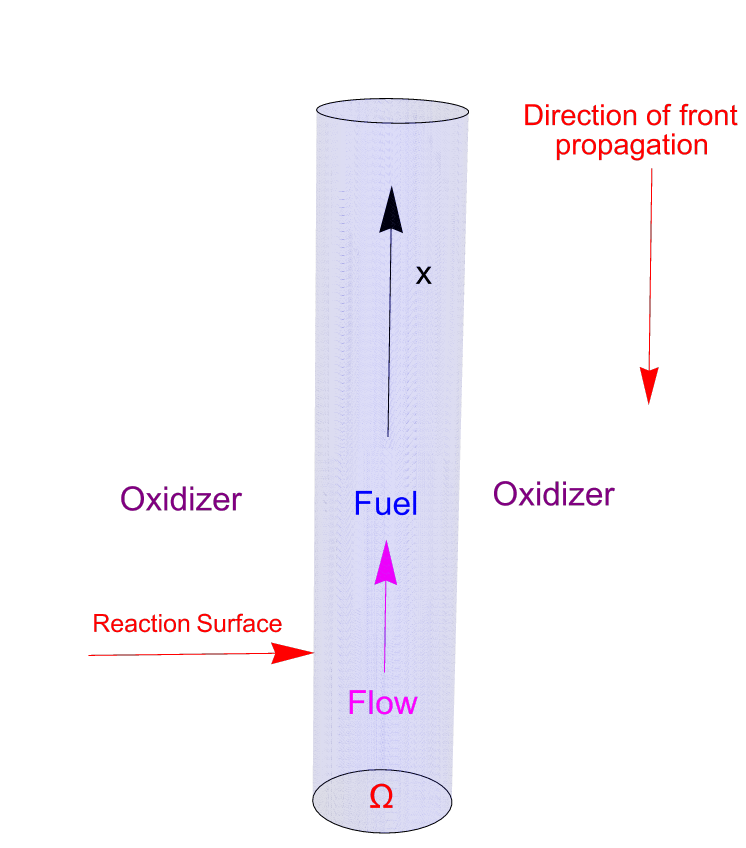}
\caption{A sketch of the reactive co-flow jet. }
\label{fig-add}
\end{figure}
 Problem \eqref{eq:i1} was studied by means of rigorous, formal, and numerical techniques in  \cite{MHHG}. Among other things, it was shown that
 solutions of \eqref{eq:i1} relevant to \textcolor{black}{the} propagation of the ignition fronts exist only \textcolor{black}{in a parametric regime} in which  the heat loss and the reaction
 strength are comparable, \textcolor{black}{see \cite[Appendix A]{MHHG}}. Specifically, there are two regimes of interest: the balanced regime  ($\tilde q \sim  \tilde h \sim O(1)$) in which diffusion, reaction
 and heat loss are of the same order of magnitude, and \textcolor{black}{the} diffusion dominated regime ($\tilde q\sim \tilde h \ll 1$) in which both reaction and heat loss are dominated by
  the diffusion. The latter regime is the subject of the current study. Formal arguments of  \cite{MHHG} show that in the diffusion dominated regime, problem \eqref{eq:i1}
  in the asymptotic limit of proportionally vanishing reaction and heat loss
  can be effectively reduced to a one-dimensional model.
  This reduction is based on the observation that in the diffusion dominated regime\textcolor{black}{,}  the temperature distribution
  in each fixed cross section of the jet
  is essentially constant, and
  hence can be well approximated by its average value. The reduced model for a traveling wave solution with speed $c$ reads:
  \begin{equation}\label{eq:i3}
 \left\{
 \begin{array}{lll}
 \theta_{xx}  -c \theta_x +q \tilde f(\theta)- h \tilde g(\theta)=0 & \mbox{in} &  \mathbb{R},\\
  \theta \to \theta_+ & \mbox{as} & x\to +\infty,\\
 \theta \to 0 & \mbox{as} & x\to - \infty,
 \end{array}
 \right.
 \end{equation}
 where $\theta(x)$ is the average of $\Theta$ over a cross section  (see \cite{MHHG} for details) and $c, h, q$ are appropriately rescaled  $\tilde c, \tilde h, \tilde q$.
 A boundary value problem \eqref{eq:i2} that  determines temperature distribution far behind the ignition front in this limit reduces to an algebraic equation:
 \begin{eqnarray} \label{eq:i4}
 q \tilde  f(s)=h \tilde g (s),
 \end{eqnarray}
 \textcolor{black}{for some $s>0$}, and $\theta_+$ is the largest solution of this equation which exists under appropriate  assumptions on $\tilde f, \tilde g, q, h$.

It was shown in \cite{MHHG} that with \textcolor{black}{a} physically feasible  choice of twice continuously  differentiable functions   $\tilde f, \tilde g,$ and parameters $q, h$,
model \eqref{eq:i1} becomes \textcolor{black}{a} well-known and well-studied problem of  traveling front propagation  with bistable nonlinearity \cite{FM}. \textcolor{black}{The choice of nonlinearities $f$ and $g$ in \cite{MHHG} (denoted by $\tilde{f}$ and $\tilde{g}$ in this paper) physically \textcolor{black}{corresponds}  to a situation  \textcolor{black}{where} the heat loss dominates the reaction for small and large temperatures, whereas the reaction dominates the heat loss for  intermediate values of \textcolor{black}{temperatures}. In view of the results of \cite{FM}, in this case \textcolor{black}{a} traveling front solution \textcolor{black}{to} problem \eqref{eq:i3}  \textcolor{black}{exists}, \textcolor{black}{is} unique (up to translation),
and is nonlinearly stable. It is also of interest, however,  to consider a physically realizable  situation \textcolor{black}{where} an interplay between the reaction and the heat loss is somewhat different from \textcolor{black}{the} one  considered in \cite{MHHG}.}
 In this paper, we consider a case in which $f(\theta)=1,$
which makes
\begin{eqnarray}\label{eq:i5}
\tilde f (\theta)=H(\theta-\theta_{ig}),
\end{eqnarray}
and $$\tilde g(\theta)=g(\theta) H (\theta-\theta_{hl}),$$ where $\theta_{hl}\in (\theta_{ig},\theta_+)$
is the heat loss threshold.   The only strict physical restriction  on $g(\theta)$ is that it is a  non-decreasing function of its argument. We will assume
that  $$g(\theta)=(1+\theta)^4-1,$$
which physically corresponds to a radiative mechanism of the heat loss.
Hence, in what follows, we set
\begin{eqnarray}\label{eq:i5}
\tilde g(\theta)=((1+\theta)^4-1)H(\theta-\theta_{hl}).
\end{eqnarray}
\textcolor{black}{From a physical perspective}, our choice of nonlinearities corresponds to  a situation in which both reaction and heat loss are absent for relatively small temperatures; when the temperature crosses ignition threshold,
the reaction switches on and remains constant and unsuppressed until the temperature rises to the values exceeding  the heat loss threshold above which both reaction and heat loss
become essential. \textcolor{black}{In view of the \textcolor{black}{above discussion,} the assumptions on nonlinearities made in this work are  clearly different from \textcolor{black}{those} made in \cite{MHHG} and physically correspond to a rather different mechanism responsible for establishing a traveling front. We also note that} though the piecewise dependency of the terms involved  in \eqref{eq:i3} adopted in this paper  is  a crude approximation,  it
 preserves major physical effects observed \textcolor{black}{in a real physical system}. A  discussion of \textcolor{black}{the} applicability of  piecewise approximation of nonlinearities in modeling reactive processes of different nature can be found in \cite{BGKS,KGGKS,MBG,MKBSG}.

The primary goal of the present work is the study of the \textcolor{black}{existence  and uniqueness}  of solution for  problem \textcolor{black}{\eqref{eq:i3}-\eqref{eq:i4}} under  the  \textcolor{black}{above-listed  assumptions on the involved nonlinearities}.
Hence, the problem under consideration for the speed $c$ and the temperature $\theta$ is as follows:
\begin{equation}\label{eq:i6}
 \left\{
 \begin{array}{lll}
 \theta_{xx}  -c \theta_x +F(\theta)=0 & \mbox{in} &  \mathbb{R},\\
  \theta \to \theta_+ & \mbox{as} & x\to +\infty,\\
 \theta \to 0 & \mbox{as} & x\to - \infty,
 \end{array}
 \right.
 \end{equation}
 where
 \begin{eqnarray}\label{eq:i7}
 F(\theta) = q H(\theta-\theta_{ig})-h((1+\theta)^4-1)H(\theta-\theta_{hl}),
 \end{eqnarray}
and
\begin{eqnarray}\label{eq:i8}
\theta_+=\left ( 1+\frac{q}{h} \right)^\frac14-1.
\end{eqnarray}
Here and below,  we will  always assume that
\begin{eqnarray}\label{eq:i9}
0<\theta_{ig}<\theta_{hl}<\theta_+,
\end{eqnarray}
which is a necessary condition for \eqref{eq:i6}-\eqref{eq:i8} to have a solution.
\textcolor{black}{A typical profile of the function} $F(\theta)$ is depicted in Figure \ref{fig:1}.
\begin{figure}[h]\vspace*{6pt}
\centering \includegraphics[width=4in]{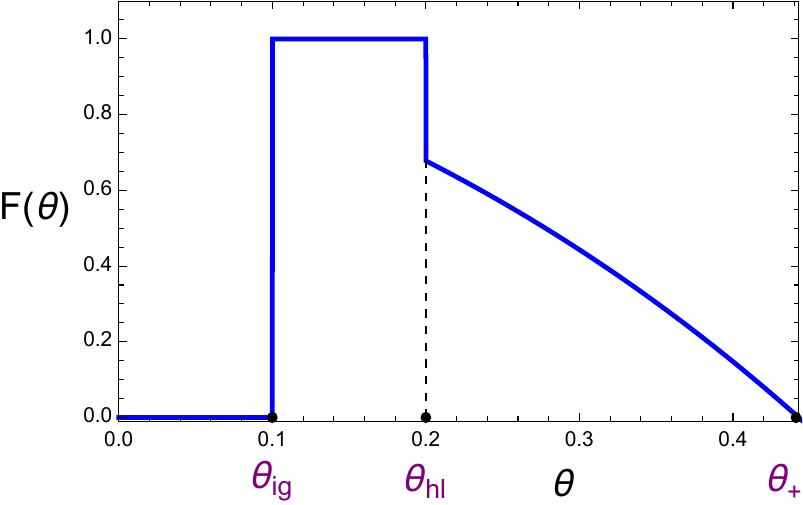}
\caption{Profile of the function $F(\theta)$ with $\theta_{ig}=0.1,~\theta_{hl}=0.2,~q=1, ~h=0.3.$ }
\label{fig:1}
\end{figure}

\section{Main results and organization of the paper}\label{s:2}

In this section, we state the main results of this paper.
We will show that problem \eqref{eq:i6}-\eqref{eq:i8} can be reformulated as
a free boundary problem and subsequently reduced to the analysis of a two-dimensional dynamical system coupled with a transcendental equation.

We look for a solution $(c,\theta)$ of \eqref{eq:i6}-\eqref{eq:i8}, where $\theta$ is a monotone non-decreasing function.
Note that any solution of \eqref{eq:i6}-\eqref{eq:i8} is translationally invariant. To fix translations, we set $\theta(0)=\theta_{ig}$ and
will \textcolor{black}{refer to the point} $x=0$ as the {\it ignition interface}. \textcolor{black}{The} monotonicity of $\theta$ also implies that there exists a unique, unknown, $R>0$
such that $\theta(R)=\theta_{hl}.$  We will refer to $x=R$ as a free interface, the {\it heat loss  interface}.
Therefore, system \eqref{eq:i6}-\eqref{eq:i8} is a free boundary or free interface problem, whose unknown is the triplet $(R,c,\theta)$.

Consequently,  any solution of \eqref{eq:i6}-\eqref{eq:i8} can be viewed as an increasing function that passes through three distinct regions: the first region $(-\infty,0)$ where the
temperature is below $\theta_{ig}$, the second region $(0,R)$ where the temperature is between $\theta_{ig}$ and $\theta_{hl}$, and the third region where the temperature exceeds $\theta_{hl}$.
In the first region, the temperature is equilibrated exclusively by diffusion, in the second region, the temperature balances diffusion and reaction whereas in the third region,
the temperature is stabilized by interplay of diffusion, reaction, and heat loss. Moreover, the temperature profile passes from region to region by preserving its continuity
and continuity of  jumps in fluxes, that is in $C^1$ manner, see Figure \ref{fig:2}.

\begin{figure}[h]
	\centering \includegraphics[width=4.6in]{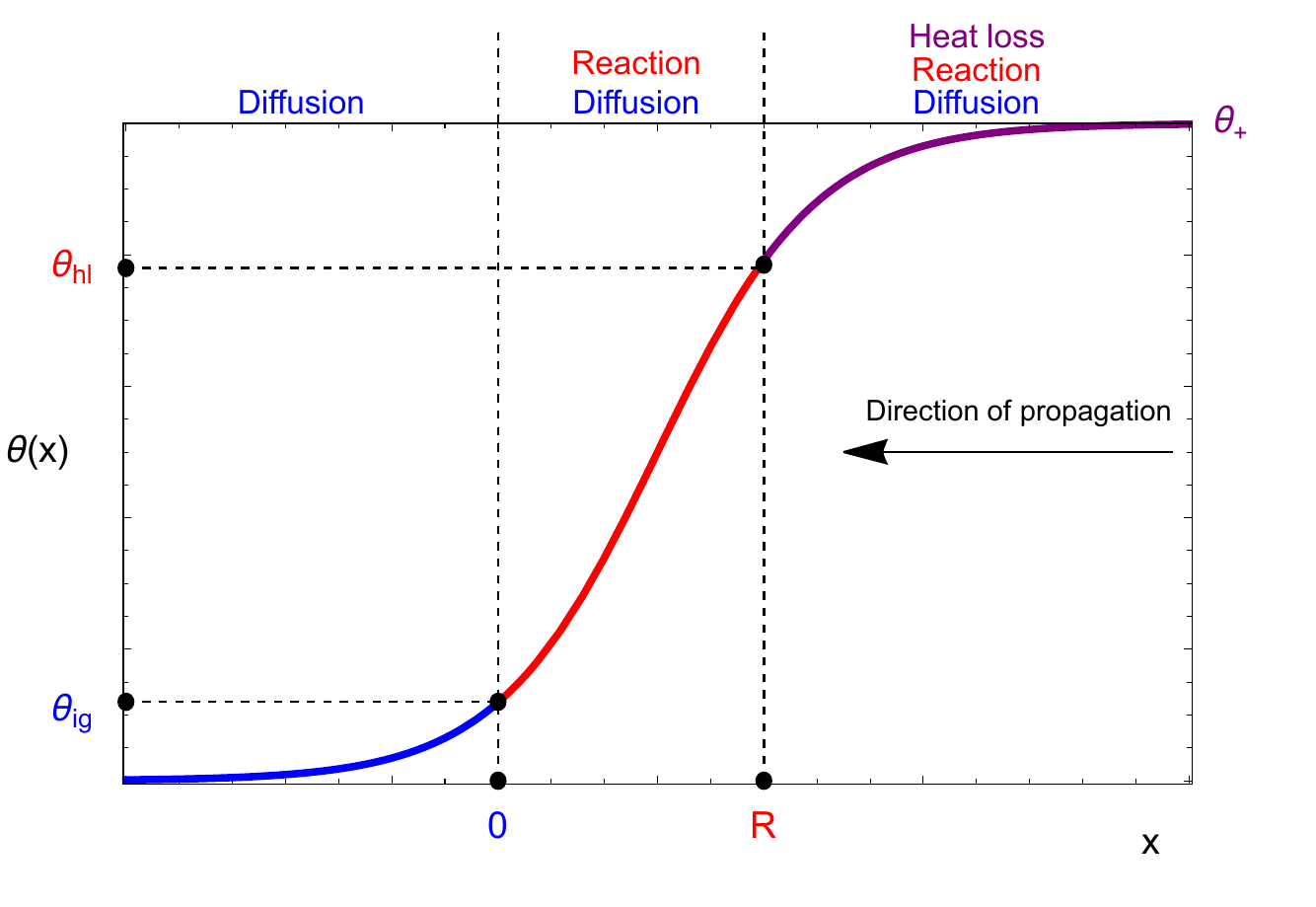}\vspace*{-8pt}
	\caption{A sketch of a  traveling front  solution for problem \eqref{eq:i6}-\eqref{eq:i8}.}
	\label{fig:2}
\end{figure}

These elementary observations allow \textcolor{black}{us} to reduce problem \eqref{eq:i6}-\eqref{eq:i8} to the following set of three problems, where it holds, respectively, $0<\theta<\theta_{ig}$, $\theta_{ig}<\theta<\theta_{hl}$ and $\theta_{hl}<\theta<\theta_{+},$
\begin{align}\label{eq:p1}
 &\left\{
 \begin{array}{ll}
 \theta_{xx}  -c \theta_x=0, &  x\in(-\infty,0),\\
 \theta \to 0 , & x\to - \infty,
 \end{array}
 \right.\\
 \label{eq:p2}
&\; \theta_{xx}  -c \theta_x+q=0, \quad  x\in(0,R),\\
\label{eq:p3}
& \left\{
 \begin{array}{lll}
 \theta_{xx}  -c \theta_x+F(\theta) =0, &  x\in(R, +\infty),\\
 \theta \to \theta_+ , & x\to +\infty,
 \end{array}
 \right.
 \end{align}
where, by abuse of language, $F(\theta)$ stands for the restriction of \eqref{eq:i7} \textcolor{black}{to} $[\theta_{hl},\theta_+]$, i.e.,
\begin{eqnarray}\label{eq:p4}
F(\theta)=q-h((1+\theta)^4-1),
\end{eqnarray}
which satisfies
\begin{align}\label{FT>0}
F(\theta)>0,\,\,\theta\in[\theta_{hl},\theta_+)
\end{align}
and has $\theta_{+}$ as the unique zero. Moreover,
\begin{align}\label{F'<0}
F'(\theta)<0,\,\, \theta\in [\theta_{hl},\theta_+].
\end{align}
In particular,
\begin{align}\label{F'theta+<0}
F'(\theta_+)<0.
\end{align}
The above system is complemented by
compatibility and
no-jump conditions \textcolor{black}{at}  the ignition and heat loss interfaces:
\begin{eqnarray}\label{eq:p5}
&&\theta(0)=\theta_{ig}, \quad \theta(R)=\theta_{hl},\\
\label{eq:p6}
&&[\theta_x]=0, \quad \mbox{at} \quad x=0 \,\,\, \text{and}\,\,\, x=R,
\end{eqnarray}
where $[\cdot]$ stands for a jump across the interface.


 Since  problems \eqref{eq:p1} and \eqref{eq:p2} are clearly linear at fixed $c>0$ and $R>0$,  their solutions satisfying  conditions \eqref{eq:p5}-\eqref{eq:p6}  at the ignition interface can easily  be obtained explicitly and read:
\begin{eqnarray}\label{eq:p7}
\theta(x)=
\left\{
\begin{array}{ll}
\theta_{ig} e^{cx}, & x\in (-\infty,0), \\
\theta_{ig} e^{cx} +\frac{q}{c} x +\frac{q}{c^2} (1-e^{cx}), & x\in(0,R).
\end{array}
\right.
\end{eqnarray}
Compatibility and no-jump conditions at the heat loss interface then require that:
\begin{align}\label{eq:p81}
&\theta(R)=\theta_{hl}=\theta_{ig} e^{cR} +\frac{q}{c} R +\frac{q}{c^2} (1-e^{cR}), \\
\label{eq:p82}
&\theta_x(R)=c\theta_{ig} e^{cR}+\frac{q}{c}\left( 1-e^{cR} \right).
\end{align}
From \eqref{eq:p81}-\eqref{eq:p82}, after some elementary algebra, we obtain
\begin{eqnarray}\label{eq:p9}
\theta_x(R)=c \theta_{hl} -q R.
\end{eqnarray}
Looking for a priori monotone non-decreasing $\theta$ requires that
\begin{eqnarray}\label{inequality}
c \theta_{hl} \geqslant qR.
\end{eqnarray}

To summarize, any solution $(R,c,\theta)$ to system \eqref{eq:p1}-\eqref{eq:p6} is fully determined by \textcolor{black}{the} solutions \textcolor{black}{to} the following problem:
\begin{equation}\label{eq:p10}
 \left\{
 \begin{array}{ll}
 \theta_{xx}  -c \theta_x+F(\theta) =0,\quad x\in(R,+\infty),\\
 \theta(R)=\theta_{hl}, \quad \theta_x(R)=c \theta_{hl} -q R,\\
 \theta\to\theta_{+},\quad x\to+\infty,
 \end{array}
 \right.
 \end{equation}
 where the relation between $c>0$ and $R>0$ is implicitly defined by the transcendental equation \eqref{eq:p81}
subjected to condition \eqref{inequality},
which will play an important role in the analysis, as we shall see.\par

The main result of this paper is as follows.
\begin{theorem} \label{t:1}
		There exists a \textcolor{black}{unique triplet $(R^{\ast}, c^{\ast}, \theta^{\ast}) \in  (\R^{+})^2 \times C^1(\R)$} that verifies system \eqref{eq:p1}-\eqref{eq:p6}.  Moreover, $\theta^*$ is increasing on $\mathbb{R}$ and there exist $K_0>0$ and $\lambda_{-}<0$ (see \eqref{eigenvalues} below) such that
		 	\begin{align}\label{exponential}
				\theta_+-\theta^{\ast}(x)\leqslant K_0e^{\lambda_-x},\:\:x\to+\infty.
			\end{align}
		
\end{theorem}

The paper is organized as follows. In Section \ref{s:3}, we  analyze the implicit relation \eqref{eq:p81} by means of the implicit function theorem and prove the following result.

\begin{theorem}\label{theorem2}
	
	Let $Q_{+}$ be the subdomain of the first quadrant \textcolor{black}{delimited by the line $(\Delta) :$ $\displaystyle R \mapsto \frac{qR}{\theta_{hl}}$, that is}
	 \begin{equation}
		Q_{+}=\big\{(R,c) : R>0,\, c\geqslant\frac{qR}{\theta_{hl}}\big\},
	\end{equation}
	and \textcolor{black}{let us} set for $(R,c)$ in 	$Q_{+}$,
	\begin{equation}\label{trans}
		G(R,c)=(e^{cR}-1-cR)q-c^2(\theta_{ig}e^{cR}-\theta_{hl}).
	\end{equation}
	The following hold:
	
(i) There exists a point $B_0=(R_{0},c_{0}) \textcolor{black}{\in (\Delta)}$ verifying $G(R_0,c_0)=0$, such that
the level set $\{(R,c) \in Q_{+} \,|\, G(R,c)=0\}$ is the graph of a smooth decreasing function $\varphi(R)$ which extends to $(0,R_0]$;\par
(ii)  The curve $\varphi(R)$ has  $B_0=(R_{0},c_{0})$ as endpoint at the right \textcolor{black}{such that $\varphi'(R_0^-)=0$,} and the axis $R=0$ as an asymptote at the left.
\end{theorem}

Section \ref{s:4} deals with the study of a generalization of  system \eqref{eq:p10}, and constitutes the core of the paper. \textcolor{black}{We adopt a topological approach in the spirit of \cite{BRSZ},} see \cite{BNS1985} for a shooting method in the case $R=0$.

For further studies, it will be convenient to  rewrite \eqref{eq:p10}  as a dynamical system. Let us set
\begin{equation}\label{notations}
t=x-R,\quad u(t)= \theta(t+R), \quad v(t)= \theta_x(t+R).
\end{equation}

In terms of these notations, problem \eqref{eq:p10}  is  equivalent to the following differential system on the plane:
\begin{equation}\label{eq:p12}
X_c : \left\{
	\begin{array}{ll}
u'(t)=v(t),\\		
v'(t)= cv(t)-F(u(t)),
 \end{array}
\right.
\end{equation}
 with initial conditions at $t=0$
\begin{equation}\label{du}
	u(0)=\theta_{hl}, \quad v(0)= c\theta_{hl}-qR
\end{equation}
\textcolor{black}
{and the asymptotic behavior
\begin{equation}\label{dul}
u \to \theta_+\quad\text{as}\quad t\to +\infty.
\end{equation}
In the system \eqref{eq:p12},
\begin{equation}\label{Fu}
F(u)=q-h((1+u)^4-1).
\end{equation}
See \eqref{FT>0}-\eqref{F'theta+<0} for its properties.
}

A distinctive  feature of system \eqref{eq:p12}, which presents a technical difficulty of its analysis,  is  that its equilibrium point $(u=\theta_{+}, v=0)$ is a  {\it hyperbolic saddle point}.
Hyperbolicity   follows from the opposite signs of   eigenvalues of the linearization of \eqref{eq:p12} about $(u=\theta_{+}, v=0)$.  These eigenvalues are:
\begin{equation}\label{eigenvalues}
	\lambda_\pm= \frac{c\pm \sqrt{c^2-4F'(\theta_+})}{2}.
\end{equation}

It is important to note that the specific choice of function $F$ has a minor impact on the dynamics of \eqref{eq:p12} as long as it has a single equilibrium point as a hyperbolic saddle.

Consequently, in what follows we will consider a more general type of nonlinearity $F$ satisfying \eqref{FT>0} and \eqref{F'theta+<0}. Namely, we consider system \eqref{eq:p12} with $F(u)$  a smooth function defined on a neighborhood $W$ of $[\theta_{hl},\theta_+]$  such that $\theta_+$ is the unique zero of $F$ on $W$, $F'(\theta_+)<0$. \textcolor{black}{Therefore, $p_0=(\theta_{+},0)$ is a hyperbolic saddle point}. Hereafter, we will refer to \eqref{Fu} as a special case.

The analysis of problem \textcolor{black}{\eqref{eq:p12}--\eqref{dul}} will utilize the following two kinds of results. The first one is the stable manifold theorem \cite{[PM]}, according to which there is a half-stable manifold $(u,v)$ depending smoothly on $c$ whose graph reads $v_{c}(u):=v(u,c)$ in $W$ and such that $v_c(\theta_{+})=0$.
The second one is  an explicit formula for evaluating the derivative of the trajectory $v_{c}(u)$ with respect to $c$ via {\it Melnikov integral} (see, e.g., \cite{[M]}). Namely, at any $\bar{u}\in [\theta_{hl},\theta_{+}], \bar{c}>0$, we have
	\begin{equation}\label{eq1}
		\partial_cv(\bar{u},\bar{c})=-\frac{1}{v(\bar{u},\bar{c})}\int _0^{+\infty}e^{-\bar{c}t}v^2_\gamma(t+t(\bar{u}),\bar{c})dt <0,
	\end{equation}
where $v_{\gamma}$ is the second component of the trajectory $\gamma_{c}(t)=(u_{\gamma}(t,c),v_{\gamma}(t,c))$ whose image is the graph of $v_c(u)$ for $u\in [\theta_{hl},\theta_{+})$; its initial condition is  $\gamma_{c}(0)=(\theta_{hl},v_{c}(\theta_{hl}))$; if $\bar{u}\in [\theta_{hl},\theta_{+})$, $t(\bar{u})$ is the time such that $u_{\gamma}(t(\bar{u}),c)=\bar{u}$.

The final result of Section \ref{s:4} is as follows.
\begin{theorem}\label{theorem3}   Consider system \eqref{eq:p12} with $F(u)$  a smooth function satisfying \eqref{FT>0} and \eqref{F'theta+<0} defined on a neighborhood $W$ of $[\theta_{hl},\theta_+]$ such that $\theta_{+}$ is the unique zero of $F$ on $W$. Then, it holds:\par
(i)	there exists a smooth mapping $\psi(R)$  sending diffeomorphically the interval $[0,+\infty)$ on the interval
	$[\psi(0),+\infty)$,  where $\psi(0)>0$; \textcolor{black}{for each $R\geqslant 0$ and $c>0$,   system \eqref{eq:p10} has a solution $\theta_{R,c}(x)$ if and only if $c=\psi(R)$; moreover, this solution is unique};\par
(ii) the curve $\psi(R)$ is increasing from $\psi(0)$ to $+\infty$ and has the line \textcolor{black}{$(\Delta) :$} $\displaystyle R \mapsto \frac{qR}{\theta_{hl}}$ as an asymptote;\par
(iii) the curve $\psi(R)$ is contained in $Q_{+}$, more precisely, for $R\geqslant0$:
	\begin{equation}
	\frac{qR}{\theta_{hl}} < \psi(R)< c_{+}(R), \quad 	c_+(R)=\frac{1}{\theta_{hl}}\Big(\big(2\int_{\theta_{hl}}^{\textcolor{black}{\theta_{+}}}F(u)du\big)^{\frac{1}{2}}+qR\Big).
	 \end{equation}
\end{theorem}
Finally, Section \ref{s:5} is devoted to the proof of Theorem \ref{t:1}. \textcolor{black}{The proof of the exponential convergence of $\theta^{\ast}(x)$ to $\theta_{+}$ as $x \to +\infty$, based on an application of Sternberg-Chen linearization theorem, is deferred to Sections \ref{A.1} and \ref{A.2} of the Appendix. In Section \ref{A.3}, we address the issue of the concavity of $\theta^{\ast}$ whenever $x \geqslant R$.}

 \section{Proof of Theorem \ref{theorem2}}\label{s:3}
The proof of Theorem \ref{theorem2} is divided into a series of lemmas and a proposition.

We consider the shooting problem
\begin{equation}\label{2}
		\begin{cases}
			\theta_{xx}-c\theta_{x}+q=0,\:\:x>0,\\
			\theta(0)=\theta_{ig},\:\theta_x(0)=c\theta_{ig}.
		\end{cases}
	\end{equation}
    Its solution is given by
    \textcolor{black}{
    \begin{align}\label{theta1}
    \theta(x)=\theta_{ig} e^{cx} +\frac{q}{c} x +\frac{q}{c^2} (1-e^{cx}),\:\:\forall x\geqslant0,
    \end{align}}
    and the derivative of $\theta$ reads
    \begin{align}\label{solution}
	\theta_{x}(x)=(c\theta_{ig}-\dfrac{q}{c})e^{cx}+\dfrac{q}{c}.
	\end{align}
	To begin with, we are interested in the range of $c$ to guarantee the existence of
\textcolor{black}{positive values of $R$} which verify $\theta(R)=\theta_{hl}$ and therefore are candidates \textcolor{black}{for} the heat loss interface.
	
	\subsection{Preliminaries}\label{prelim}
Let us define:
\begin{equation} \label{A}
	a=\frac{1}{b}( \frac{\theta_{hl}}{\theta_{ig}} -1 )>0, \;\; b= \sqrt{\frac{q}{\theta_{ig}}}.
\end{equation}


The point with coordinates $(a,b)$ will be denoted by $A$ below.
\textcolor{black}{\begin{remark}\label{remarkA}
The motivation of introducing the values $a$ and $b$ is to get an explicit solution of the transcendental equation $G(R,c)=0$. To do that, $b$ is chosen such that the exponential term is eliminated. Then, $a$ is the unique solution to the resulting equation. Therefore, $A=(a,b)$ is the only explicit solution to the equation $G(R,c)=0$. It is worth noting that point $A$ will play a significant role in the implementation of the implicit function theorem, see Lemma \ref{IFT} below.
\end{remark}}
\textcolor{black}{In the following, we explore the condition for $c$ to guarantee the existence of a unique $R$ such that $\theta(R)=\theta_{hl}$ and $\theta_x(R)\geqslant 0$, where $\theta$ is defined in \eqref{theta1}.
    \begin{lemma}\label{cR}
    There exists $c_0>0$ such that when $c=c_0$, we have a unique $R_0$ satisfying $\theta(R_0)=\theta_{hl}$ and $\theta_x(R_0)=0$. For any given $c> c_0$, there exists a unique $R>0$ such that $\theta(R)=\theta_{hl}$ and $\theta_x(R)> 0$. 
    \end{lemma}
    \begin{proof}
    First we prove the existence of $c_0$. Denote
\begin{equation}\label{eqec}
\kappa(c):=(1-\frac{\theta_{ig}}{q}c^2)e^{\frac{\theta_{hl}}{q}c^2}-1.
\end{equation}
It is easy to get from \eqref{eqec} that
    \begin{align}\label{ed}
    \kappa'(c)=2ce^{\frac{\theta_{hl}}{q}c^2}(\frac{\theta_{hl}-\theta_{ig}}{q}-\frac{\theta_{ig}\theta_{hl}}{q^2}c^2).
    \end{align}
    The unique root of \eqref{ed} is
    \begin{equation}
    \tilde{c}=
	\big(\frac{q(\theta_{hl}-\theta_{ig})}{\theta_{hl} \theta_{ig}}\big)^{\frac{1}{2}}.
\end{equation}
  Therefore, $\kappa(c)$ is increasing for $0<c<\tilde{c}$ and decreasing for $\tilde{c}<c<b$. As
    $\kappa(0)=0$ and
    $\kappa(c)\to -\infty$ as $c\to +\infty$,
    there exists a unique $c_{0}>0$ such that
    \begin{equation}\label{ec00}
    \kappa(c_0)=0.
    \end{equation}
    Noticing $\kappa(b)=-1$, we obtain that $\tilde{c}<c_0<b$.
    Thus,
      \begin{equation}\label{ic0}
    \kappa(c)\leqslant 0,\quad \forall  c_{0}\leqslant c<b.
    \end{equation}
    Next, we consider the following three cases:\\
    (i) If $c\geqslant b$, i.e., $c\theta_{ig}\geqslant \frac{q}{c}$, then $\theta_{x}(x)\geqslant \frac{q}{c}>0$ and $\theta(x)\to +\infty$ as $x\to +\infty$. Thus, there exists a unique $R$ with $\theta(R)=\theta_{hl}$ and $\theta_{x}(R)>0.$\\
    (ii) If $c_0< c<b$, which gives $c\theta_{ig}< \frac{q}{c}$, then there exists a unique $x_0=\frac{1}{c}\ln (\frac{q}{q-c^2\theta_{ig}})\break>0$ such that $\theta(x)$ is increasing for $0<x<x_0$ and decreasing for $x>x_0$. Moreover, $$\theta(x)\leqslant \theta(x_0)=\frac{q}{c}x_0,\quad \forall x\geqslant 0.$$
   From \eqref{ec00}-\eqref{ic0}, one has
    $$\kappa(c)=(1-\frac{\theta_{ig}}{q}c^2)e^{\frac{\theta_{hl}}{q}c^2}-1< 0,\quad \forall c_0<c<b,$$
    which is equivalent to
    $$\theta_{hl}< \theta(x_0),\quad \forall c_0< c<b.$$
    Thus, there exist $R_1\in (0,x_0)$ and $R_2\in(x_0,+\infty)$ such that
    \begin{align*}
    \theta(R_1)=\theta(R_2)=\theta_{hl}, \,\,\,\theta_x(R_1)>0\,\,\, \text{and}\,\,\, \theta_x(R_2)<0.
    \end{align*}
    We can thus choose $R=R_1$.\\
    (iii) If $c=c_0$, from \eqref{ec00}, one has
    $$
    \kappa(c_0)=(1-\frac{\theta_{ig}}{q}c_0^2)e^{\frac{\theta_{hl}}{q}c_0^2}-1=0,
    $$
    which is equivalent to
    $$
    \theta_{hl}= \theta(x_0).
    $$
    Thus, $R_{0}=x_0$ and $\theta_x(R_{0})=0$. Finally, integrating the equation in (\ref{2}) from $0$ to $R_{0}$ yields a relation between $R_0$ and $c_0$ that reads
    \begin{align}\label{relation_C0R0}
    	c_0\theta_{hl}=qR_{0}.
    \end{align}
    \end{proof}}
    \subsection{Application of the implicit function theorem}
     Let $R>0$ be given in Lemma \ref{cR}. Then, as we have seen, the compatibility and no-jump conditions provide a transcendental relationship between $R$ and $c$ in (\ref{eq:p81}) or equivalently defined by $G(R,c)=0$, where $G(R,c)$ is given by \eqref{trans}.
	
	\textcolor{black}{Let us recall that} $Q_{+}=\{(R,c) : R>0, c\geqslant\frac{qR}{\theta_{hl}}\}$; clearly, the mapping $(R,c) \mapsto G(R,c)$ is $C^{\infty}$ on $Q_{+}$. 
	\textcolor{black}{Hereafter, we denote by $\interior({Q}_{+})$ the interior of ${Q}_{+}$.}
    \begin{lemma}\label{dG}
  (i) If $(R,c) \in \textcolor{black}{\interior({Q}_{+})}$ satisfies $G(R,c)=0$, then $\frac{\partial G}{\partial R}(R,c)<0$ and $\frac{\partial G}{\partial c}(R,c)<0$;\par
    (ii) if $(R,c)  \in \textcolor{black}{\interior({Q}_{+})}$ satisfies $cR>\frac{\theta_{hl}-\theta_{ig}}{\theta_{ig}}$, then $\frac{\partial G}{\partial c}(R,c)<0$.
    \end{lemma}
    \begin{proof}
    (i) If $(R,c)$ satisfies $G(R,c)=0$, since $c\theta_{hl}>qR$, one has
    \begin{align*}
    &\frac{\partial G}{\partial R}(R,c)=c^2(qR-c\theta_{hl})<0,\\
    &\frac{\partial G}{\partial c}(R,c)=cR(qR-c\theta_{hl})-\frac{2q}{c}(e^{cR}-1-cR)<0.
    \end{align*}\par
    (ii) If $(R,c)$ satisfies $cR>\frac{\theta_{hl}-\theta_{ig}}{\theta_{ig}}$, thanks to $c\theta_{hl}>qR$,
    \begin{align*}
    \frac{\partial G}{\partial c}(R,c)&=qR(e^{cR}-1)-c\theta_{ig}e^{cR}(2+cR)+2c\theta_{hl}\\
    &<c((\theta_{hl}-2\theta_{ig}-cR\theta_{ig})e^{cR}+\theta_{hl}).
    \end{align*}
    Since
    $$\theta_{hl}-2\theta_{ig}-cR\theta_{ig}<\theta_{hl}-2\theta_{ig}-\frac{\theta_{hl}-\theta_{ig}}{\theta_{ig}}\theta_{ig}=-\theta_{ig}$$ and $e^{cR}>cR+1$, we obtain
    \begin{align*}
    \frac{\partial G}{\partial c}(R,c)<c(-\theta_{ig}(cR+1)+\theta_{hl})<c(-\theta_{ig}(\frac{\theta_{hl}-\theta_{ig}}{\theta_{ig}}+1)+\theta_{hl})=0.
    \end{align*}
    \end{proof}

    \begin{lemma}\label{(0,+)}
    Let $(R,c)\in Q_{+}$ such that $G(R,c)=0$. Then, $c\rightarrow +\infty$ as $R\rightarrow 0^{+}$.
	\end{lemma}
    \begin{proof}
    \textcolor{black}{Let us set $m(R)=cR>0$ and
    \begin{align*}
    \liminf_{R\to 0^{+}}m(R)=L,
    \end{align*}
    where $0\leqslant L \leqslant +\infty$.
    Then, $m(R)$ verifies
    \begin{align}\label{m}
    R^{2}(e^{m(R)}-1-m(R))q=m^{2}(R)(\theta_{ig}e^{m(R)}-\theta_{hl}),
    \end{align}
    and there exists a sequence $\{{R}_{k}\}$ such that
      \begin{align*}
      \lim_{{R}_{k}\to 0^{+}}m({{R}}_{k})=L.
      \end{align*}}
  \textcolor{black}{
      Next, we prove $L\neq 0$. Let us assume by contradiction that $L=0$, then $m({{R}}_{k})\to 0$ as ${R}_{k}\to 0^{+}$. By Taylor formula, the left-hand side of (\ref{m}) reads $$\frac{q}{2}{{R}}_{k}^2m^2({{R}}_{k})+o(m^2({R}_{k}))$$
      and the right-hand side is
      $$(\theta_{ig}-\theta_{hl})m^2({R}_{k})+o(m^2({R}_{k})),$$ which is ruled out as ${
      R}_{k}\to 0^{+}$. Therefore, one can obtain
      $$\displaystyle \liminf_{R\rightarrow 0^+}m(R)=L>0.$$
     Finally,
     \begin{align*}
     \displaystyle c=\frac{m(R)}{R}\rightarrow +\infty\,\,\, \text{as}\,\,\, R\rightarrow 0^{+}.
     \end{align*}}
    \end{proof}


   	 We will use the implicit function theorem to prove the existence and uniqueness of a curve $G(R,c)=0$ contained in $Q_{+}$.  It is easy to verify that the point $A = (a,b)$ (see \eqref{A}) belongs to $Q_{+}$. The idea (see also \cite{BSZ}) is to apply the implicit function theorem at $A$ and prove the existence of local branches of the curve in both directions; then, extend them.
   	
   	As discussed in Remark \ref{remarkA}, the point $A =(a, b)$ plays a significant role in our analysis. Another significant point is the point with coordinates $(R_0, c_0)$, which we will denote by $B_0$ below. From the proof of Lemma \ref{cR}, we know that $0<a<R_{0}$, $c_{0}<b$, and $G(R_{0},c_{0})=0$. Furthermore, the relation \eqref{relation_C0R0} shows that $B_0$ belongs to the line $(\Delta)$ for $c_0\theta_{hl} = qR_0$. Therefore, $B_0$ is included in $Q_+$, and more specifically, on the boundary of $Q_+$, see Figure \ref{Fig_c0new}.
   	
    \begin{lemma}\label{IFT}
    There exists a neighborhood $U$ of $A=(a,b)$ in $Q_{+}$, an open interval $I\subset (0,R_{0})$ such that $a \in I$, and a real valued function $\widetilde{\varphi}:I\rightarrow \mathbb{R}$ such that, for any $R\in I$, $(R, \widetilde{\varphi}(R))\in U$ and
    $$G(R,\widetilde{\varphi}(R))= 0, \quad \widetilde{\varphi}(a)=b.$$
    Moreover, $\widetilde{\varphi}(R)$ is smooth on $I$ and monotonically decreasing with respect to $R$.
    \end{lemma}
    \begin{proof}
    As $G(a,b)=0$, $\frac{\partial G}{\partial c}(a,b)<0$ by Lemma \ref{dG}. Moreover, $G(R,c)$, $\frac{\partial G}{\partial R}(R,c)$ and $\frac{\partial G}{\partial c}(R,c)$ are smooth in $Q_{+}$. Thanks to the implicit function theorem, $G(R,c)=0$ uniquely determines a $C^{\infty}$ function $c=\widetilde{\varphi}(R)$ on an open interval $I$. Indeed, the function $\widetilde{\varphi}(R)$ is such that $b=\widetilde{\varphi}(a)$, and $G(R,\widetilde{\varphi}(R)) = 0$ in a neighborhood $U$ of $A$. By Lemma \ref{dG}, it holds that
    $$\frac{d\widetilde{\varphi}}{dR}(R)=-\frac{\partial G}{\partial R}(R,\widetilde{\varphi}(R))\left(\frac{\partial G}{\partial c}(R,\widetilde{\varphi}(R))\right)^{-1}<0.$$
    \end{proof}
    In the following proposition, we  prove the global existence of the curve defined by $G(R,c)=0$ on $Q_{+}$ by extending the local \textcolor{black}{branches}.
    \begin{proposition}\label{varphi}
    The level set $\{(R,c) \in Q_{+} \,|\, G(R,c)=0\}$ is the graph of a decreasing, $C^{\infty}$ function $\varphi(R)$
    $$\varphi:\left(0,R_{0}\right]\rightarrow \left [c_{0},+\infty\right).$$
   \textcolor{black}{Moreover}, the curve $\varphi(R)$ has  $B_0=(R_{0},c_{0})$ as endpoint at the right \textcolor{black}{such that $\varphi'(R_0^-)=0$,} and the axis $R=0$ as an asymptote at the left (see Figure \ref{Fig_c0new}).
    \end{proposition}

    \begin{figure}[h]
    	\centering \includegraphics[width=4in, trim=0mm 10mm 0mm 0mm, clip]{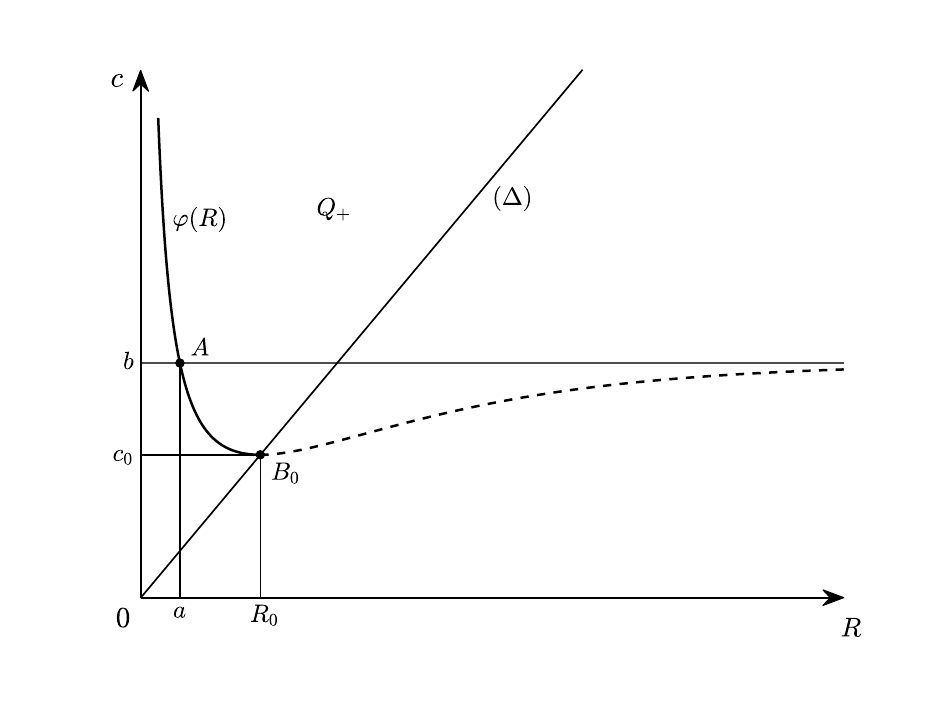}
    	\caption {Profile of the implicit function $\varphi(R)$ \textcolor{black}{which intersects the line $(\Delta)$ at} $B_0=(R_0,c_0)$ \textcolor{black}{with a horizontal tangent.} Here, $q = 1,\, \theta_{ig} = 0.2,\,\mbox{and}\; \theta_{hl} = 0.25.$}
    	\label{Fig_c0new}
    \end{figure}

    \begin{proof}
    Since $G(R,c)$ and $\frac{\partial G}{\partial c}(R,c)$ are continuous on $Q_{+}$, for any $(c,R)\in \textcolor{black}{\interior(Q_{+})}$ satisfying $G(R,c)=0,$ we have
    $$\frac{\partial G}{\partial c}(R,c)<0.$$ We can reiterate the implicit function theorem and extend the local branch in Lemma \ref{IFT}.
    The proof is divided in two parts about the right and the left extensions, respectively.

    (i) {\it Right extension:} assume the endpoint on the right is a point $(R_{r},c_{r})$ and $G(R_r,c_r)=0$ by continuity. The maximum right extension $\widetilde{\varphi}_{r}(R)$ is a $C^{\infty}$ decreasing function
    $$\widetilde{\varphi}_{r}:\left[a,R_r\right)\rightarrow \left(c_r,b\right].$$
    Let us prove that $(R_r,c_r)=(R_{0},c_{0})$.

    If $0<R_r<R_{0}$, $c_r>c_{0}$, since $G(R_r,c_r)=0$, we are in position to use the implicit function theorem again at $(R_r,c_r)$, which contradicts with the maximality of the branch $\widetilde{\varphi}_{r}$.

    If $0<R_r<R_{0}$, $c_r=c_{0}$, then $G(R_r,c_{0})=0$. But, from Lemma \ref{cR}, we have that for $c=c_{0}$, $R_{0}$ is such that $G(R_{0},c_{0})=0$ is unique. Hence, this case is ruled out.

    If $R_r=R_{0}$, $c_r>c_{0}$, then $G(R_{0},c_r)=0$. \textcolor{black}{The proof of Lemma \ref{cR}} shows that $c_{0}\theta_{hl}=qR_{0}$ and  $c_{0}> \tilde{c}$, thus
    \begin{align*}
     cR_{0}>c_{0}R_{0}=\frac{c_{0}^2\theta_{hl}}{q}>\frac{\tilde{c}^{\:2}\theta_{hl}}{q}=\frac{\theta_{hl}-\theta_{ig}}{\theta_{ig}}\,\,\, \text{for}\,\,\, c>c_{0}.
     \end{align*}
      From Lemma \ref{dG}, for fixed $R_{0}>0$, it holds that $\frac{\partial G}{\partial c}(R_{0}, c)<0$ for all $c>c_{0}$. This contradicts with $G(R_{0},c_{0})=0$. Therefore, we proved that $(R_r,c_r)=(R_{0},c_{0})$ \textcolor{black}{and the maximum right extension  $\widetilde{\varphi}_r$}
      $$\widetilde{\varphi}_r:\left[a,R_0\right)\rightarrow \left(c_0,b\right]$$
    may be \textcolor{black}{smoothly} extended to $[a,R_0]$ by continuity and intersects the line $(\Delta)$ \textcolor{black}{at $B_0=(R_0,c_0)$ because $c_0\theta_{hl}=qR_0$. Finally, we observe that the left derivative of $\widetilde{\varphi}_r$ is given by the formula
      $$\frac{d\widetilde{\varphi}_r}{dR}(R_0^-)=-\frac{\partial G}{\partial R}(R_0,\widetilde{\varphi}_r(R_0))\left(\frac{\partial G}{\partial c}(R_0,\widetilde{\varphi}_r(R_0))\right)^{-1}=0.$$
      (see the proof of Lemma \ref{dG} (i)). Therefore, the curve $\widetilde{\varphi}_r$ intersects the line $(\Delta)$ at $B_0$ with a horizontal tangent, see Figure \ref{Fig_c0new}}.\par

    (ii) {\it Left extension:} We prove that for the maximum left extension $\widetilde{\varphi}_{\ell}(R)$,
    \begin{align}\label{c0l}
    \lim_{R\rightarrow 0^{+}}\widetilde{\varphi}_{\ell}(R)=+\infty,
    \end{align}
    which implies that the axis $R=0$ is an asymptote of the curve defined by $G(R,c)=0$ in $Q_{+}$.\par
    Let us first show that $\widetilde{\varphi}_{\ell}(R)$ is well-defined on $(0,a)$. We argue by contradiction and assume that there is an endpoint $(R_{\ell},c_{\ell})$ on the left. Hence, $G(R_{\ell},c_{\ell})=0$ by continuity. If $0<R_{\ell}<a$, we may use the implicit function theorem again, which contradicts with the maximality of $\widetilde{\varphi}_{\ell}(R)$. If $R_{\ell}=0$, then $G(0,c_{\ell})=0$, which is obviously excluded by Lemma \ref{(0,+)}.\par
    According to Lemma \ref{IFT}, $\widetilde{\varphi}_{\ell}(R)$ is a decreasing function on $(0,a)$. Let us assume that (\ref{c0l}) does not hold. Thus, one has a sequence $\{R_{k}\}$ verifying $R_{k}\neq 0$ and $R_{k}\to 0^+$, such that
    \begin{align}
    \widetilde{\varphi}_{\ell}(R_{k})\leqslant M,
    \end{align}
    for some constant $M>0$. Therefore, $\{\widetilde{\varphi}_{\ell}(R_{k})\}$ is bounded and we can extract a subsequence $\{R_{k'}\}$ such that $\{\widetilde{\varphi}_{\ell}(R_{k'})\}$ converges to some constant $C>0$ as $k' \to +\infty$. Since
    \begin{align}
    G(R_{k'},\widetilde{\varphi}_{\ell}(R_{k'}))=0,
    \end{align}
   and $\widetilde{\varphi}_{\ell}(R)$ and $G(R,c)$ are continuous functions, letting $R_{k'} \to 0^+$, we get
    \begin{align}
    G(0,C)=0.
    \end{align}
    This contradicts with Lemma \ref{(0,+)}.

    Finally, the local branch in Lemma \ref{IFT} is extended to the whole interval $\left(0,R_0\right]$, and we denote it by $\varphi(R)$.
    \end{proof}

\section{Dynamical system and proof of Theorem \ref{theorem3}}\label{s:4}

We consider the system \eqref{eq:p12} with $F(u),$  a smooth function satisfying \eqref{FT>0} and \eqref{F'theta+<0}. Then, we can choose an open interval $W$ containing $[\theta_{hl},\theta_+],$ such that $\theta_{+}$ is the unique zero of $F$ on $W$. Let us write $X_c$ the vector field defined by \eqref{eq:p12}. We look at $X_c$ on the strip $W\times \R^+$ (here $\R^+=[0,+\infty)$). Let us notice that $F(u)>0$ on $W\cap \{u<\theta_+\}.$
\vskip10pt
We have that
$$X_c=X_0+cv\frac{\partial}{\partial v},$$ where $X_0$ is the Hamiltonian vector field corresponding to the Hamiltonian function $$\mathcal{H}(u,v)=\frac{1}{2}v^2+ \mathcal{U}(u).$$ Here,
$$\mathcal{U}(u)=\int_{\theta_{hl}}^uF(s)ds.$$

On  $W\times \R^+,$ the vector field $X_c$ has a unique singular point $p_0=(\theta_+,0).$
This singular point is a hyperbolic saddle point with eigenvalues
$$\lambda_\pm= \frac{c\pm \sqrt{c^2-4F'(\theta_+})}{2}.$$
\begin{figure}[h]
	\centering \includegraphics[width=4.5in]{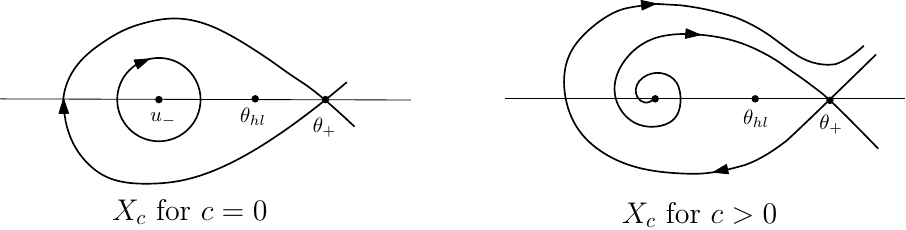}
	\caption{Phase portraits of $X_{0}$ and $X_{c}$,  $0<c<2\sqrt{F'(u_-)}$.}
	\label{FigXc}
\end{figure}

In the special case \eqref{Fu}, the vector $X_{c}$ is defined on the entire plane $\R \times \R$. The function $F(u)$ has two roots on $\R$:
\begin{equation}
	u_{\pm} =-1 \pm \Big(\frac{q+h}{h}\Big)^{\frac{1}{4}},
\end{equation}
which are simple. One is $u_{+}=\theta_{+}>0$ (see \eqref{eq:i8}), but $u_{-}$ is negative, with no physical meaning. $X_{c}$ has two singular points, a saddle $(\theta_{+},0)$, \textcolor{black}{and a singular point $(u_{-},0)$, where $F'(u_{-})>0$; this point is a center if $c=0$, an unstable focus if $0<c<2\sqrt{F'(u_{-})}$,  and an unstable node if $c>2\sqrt{F'(u_{-})}$. }

Drawing the phase portrait of $X_{0}$ is trivial, knowing that the trajectories are contained in the level curves of $\mathcal{H}(u,v)$. For $c>0$, the phase portrait can be easily obtained, with the help of Poincar\'e-Bendixson theorem (see \cite{[PM]}). These phase portraits are pictured in Figure \ref{FigXc}, where the one on the right shows the case when $0<c<2\sqrt{F'(u_{-})}$.

\subsection{\bf Existence of the function $v_c(u)$}\label{4.1}

The saddle point $p_0$ has a stable and an unstable \textcolor{black}{manifold} of dimension one.  We are interested in the half part of the stable manifold located in the region $\mathcal{W}=
[\theta_{hl},\theta_+]\times \R^+ \subset W\times \R^+.$  We \textcolor{black}{refer to} this half-stable manifold \textcolor{black}{as} the stable \textcolor{black}{separatrix}  $S_c$ (of $p_0$). Note that $S_c$  is the union of an orbit $\gamma_c$ of $X_c$ tending toward $p_0$ and this limit point.

In our particular case, the stable manifold  theorem  (see \cite{[PM],Irwin})
says that $S_c$ is a $c$-family of smooth  curves depending smoothly on $c$ (see, e.g., \cite[Theorem 6.2, p. 75]{[PM]}).
Let us make this claim more precise.
\vskip2pt
a. First, for $c=0$,  the stable \textcolor{black}{separatrix} $S_0$ is contained in the level $$\mathcal{H}(u,v)=\mathcal{H}(\theta_+,0)=\mathcal{U}(\theta_+).$$ As a consequence,  $S_0$ is the graph:
$$\{v=v_0(u)=\sqrt{2(\mathcal{U}(\theta_+)-\mathcal{U}(u))}\},$$  for $u\in [\theta_{hl},\theta_+].$
\vskip2pt
b. For $c>0$, let us consider now the topological closed disk: $$\mathcal{B}=\{(u,v) | \theta_{hl}\leqslant u \leqslant \theta_+ , 0\leqslant v\leqslant v_0(u)\},$$
whose boundary $\partial\mathcal{B}$ is the topological triangle $[p_0,p_1,p_2]$  of corners $p_0,p_1$, and $p_2$ (see Figure
\ref{Figgamma}).
\begin{figure}[h]
	\centering \includegraphics[width=4.4in]{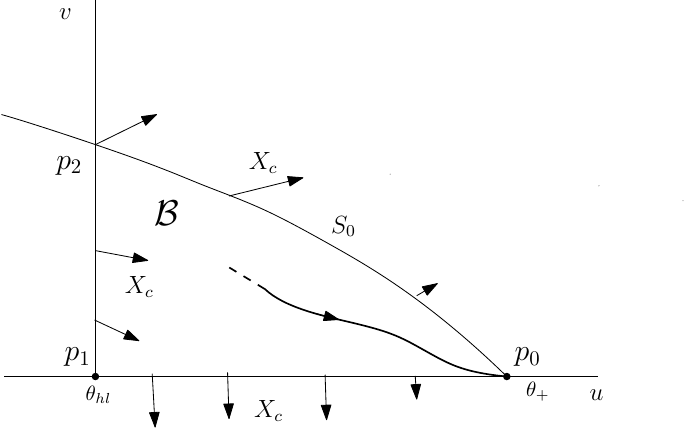}
	\caption{The vector field $X_c$ in $\mathcal{B}$.}
	\label{Figgamma}
\end{figure}

The vector field $X_c$ points inside $\mathcal{B}$ along the left vertical side $[p_1,p_2]$  of $\partial\mathcal{B}$ and points outside  $\mathcal{B}$ along the two other sides  of the boundary: The horizontal side $[p_1,p_0]$
contained in $\{v=0\}$  and the curved side $[p_2,p_0]$ contained in
$S_0.$ Let us consider now the orbit $\gamma_c.$ It is the image of a trajectory $\gamma_c(t)$ with an intial point $p=\gamma_c(0)$ in the interior of
$\mathcal{B}.$ We know that $\gamma_c(t)\rightarrow p_0$ as $t\rightarrow
+\infty.$  For the negative $t$, we have the following alternative, easy consequence of what is said above  the $X_c$ along $\partial\mathcal{B}$:

{\it There exists a  maximal $t_1<0$ such that $\gamma_c(t_1)\in [p_1,p_2]$ or the $\gamma_c(t)$ is defined for all negative $t$ and the $\alpha$-limit set of $\gamma_c$ (the set of adherent points for $t\rightarrow -\infty$) is a non-empty compact set $L,$  contained in the interior of $\mathcal{B}.$}

If the second alternative holds, then we could apply the Poincar\'e-Bendixson theorem. A general version of this theorem may be found in \cite{[PM]}; however, in the present article, we use a weak version given in detail in \cite{[R]}. The consequence is that the set $L$ must contain a singular point or a closed orbit of $X_c$. In the latter case, since $\mathcal{B}$ is simply connected, there also exists a  singular point in the interior of $\mathcal{B}$. Then, in any case, the interior of $\mathcal{B}$ must contain a singular point, which is  impossible.
\begin{remark}
The idea of applying the Poincar\'e-Bendixson theorem in order to study a particular orbit was already used  in \cite{BRSZ}.
The problem under consideration herein is simpler: We needed to determine the trajectory in \cite{BRSZ}, whereas we already know which trajectory to take into account in the present work,
namely $\gamma_c(t)$. 
We just want to prove that $\gamma_c(t)$ reaches the side $[p_1,p_2]$ on $\partial\mathcal{B}$ \textcolor{black}{at some} negative time.
\end{remark}

Then, the first term of the alternative is true:  the trajectory $\gamma_c(t)$ starts at a point $(\theta_{hl},v_c(\theta_{hl}))$ with $$0<v_c(\theta_{hl})<v_0(\theta_{hl}).$$ But, as $\gamma_c$ is located in $\{v>0\}$,
$S_c$ must be a graph $v=v_c(u)$ above $u\in[\theta_{hl},\theta_+].$ The trajectory $\gamma_{c}(t)$ may be written as
\begin{equation}\label{gamma_c}
\gamma_c(t)=(u_{\gamma}(t,c),v_{\gamma}(t,c))
\end{equation}
in the coordinates $(u,v)$. Making a time translation, we assume now that
\begin{align*}
\gamma_{c}(0)=(\theta_{hl},v_{c}(\theta_{hl}))\quad \text{and} \quad \gamma_c(t) \to (\theta_{+},0)\,\, \text{as} \,\, t \to +\infty.
\end{align*}
 The function $v_{c}(u)$ is obtained by eliminating the time $t$ between $v_{\gamma}(t,c)$ and $u_{\gamma}(t,c)$.

If we return to the original notations (see \eqref{notations}), we have to make the time change $x=R+t$ and the $u$-component $u_{\gamma}(t,c)$ \textcolor{black}{gives a $\theta$-solution as $\theta_{R,c}(x)= u_{\gamma}(x-R,c)$ of the following equation:
\begin{equation}\label{*}
 \left\{
 \begin{array}{ll}
 \theta_{xx}  -c \theta_x+F(\theta)=0, \:\:\: x\in[R,+\infty),\\
 \theta(R)=\theta_{hl},\:\:\: \theta(x) \to \theta_+ \:\:\text{as}\:\: x\to + \infty.
 \end{array}
 \right.
 \end{equation}
Note that \eqref{*} is just (22), without the condition $\theta_{x}(R)=c\theta_{hl}-qR$; the latter condition will be treated in Section \ref{4.3}.} We can formulate this as follows.

\begin{proposition}\label{theta_R_c}
For each pair $(R,c) \in (\R^{+})^2$, there exists a unique solution $\theta_{R,c}(x)$ of \textcolor{black}{\eqref{*}} such that
\begin{align*}
\theta_{R,c}(R)=\theta_{hl}\quad \text{and} \quad \theta_{R,c} \to \theta_{+}\,\, \text{as} \,\, x\to +\infty.
\end{align*}
 Moreover, it holds that $\frac{d}{dx} \theta_{R,c}>0$ for $x\in [R,+\infty)$. The solution is unique according to the Cauchy theorem.

\end{proposition}
\begin{proof}
As mentioned above, we have $\theta_{R,c}(x)= u_{\gamma}(x-R,c)$. Then, it follows from  \eqref{notations}-\eqref{eq:p12} that
\begin{equation}
	\frac{d}{dx} \theta_{R,c}=v_{\gamma}(x-R,c)=\frac{du_{\gamma}}{dt}(x-R,c).
\end{equation}
As the orbit $\gamma_{c}$ is contained in $\{v>0\}$, we have that $v_{\gamma}(x-R,c)>0$ for $x\in [R,+\infty)$.
\end{proof}
The graph of the function $v_c(u)$, defined above, is a compact subinterval on the stable manifold of the saddle point $p_0=(\theta_+,0)$ of the vector field $X_c$. Then, the stable manifold theorem implies that $v_c(u)$ is a smooth function of $(u,c)$ on
$[\theta_{hl},\theta_+]\times \R^+$.

The \textcolor{black}{separatrix} property of $S_c$ is expressed by the fact that $v_c(\theta_+)=0$ for all $c$ (but, of course, $v_c(u)>0$ for $u<\theta_+).$

\begin{remark}(i) In \cite{[PM]}, the stable manifold theorem is only proved for a hyperbolic saddle point $p$ of a diffeomorphism. But, if $p$ is a hyperbolic saddle of a vector field $X,$ $p$ is also a hyperbolic saddle of $\varphi_t$  for any $t\not =0,$ where $\varphi_t$ is a one-parameter group of $X$. Consider now the stable manifold $\mathcal{V}$ of $\varphi_1.$
	As $\varphi_t$ commutes with $\varphi_1$ for any $t,$   we have that $\varphi_t(\mathcal{V})$ is also a stable manifold of $\varphi_1.$ As $\mathcal{V}$ is the unique stable manifold of $\varphi_1$
	at  $p$, we have that
	$\varphi_t(\mathcal{V})=\mathcal{V}$ for all $t.$ This is equivalent to say that $X$ is tangent to $\mathcal{V}$, and thus is the stable manifold of $X$ at $p.$
	
	(ii) \cite[Theorem 6.2]{[PM]} is only proved for a small compact disk, neighborhood of $p$ on the stable manifold $\mathcal{V}_c.$  But, it can be easily extended to any  compact disk, neighborhood of $p$ on $\mathcal{V}_c$ (for  such a disk $\mathcal{V}_c$ containing the \textcolor{black}{separatrix} $S_c$ in our case; we can use the Cauchy theorem at regular points of $X_c$ in order  to trivially obtain an extension to the whole \textcolor{black}{separatrix} $S_c$).
\end{remark}

\subsection{\bf Computation of  $\partial_cv(u,c) $}

In this subsection, we choose  to write
\begin{align*}
v_c(u)=v(u,c), \quad \frac{\partial v_c}{\partial c}(u)=\partial_cv(u,c),\quad  \frac{\partial X_c}{\partial c}=\partial_cX_c.
\end{align*}
As
\begin{align*}
{\rm Det}(X_c,\partial_cX_c)=v^2>0\,\,\, \text{for} \,\,\,v>0,
\end{align*}
 the family of vector fields $X_c$ has the {\it rotation property} (see \cite{Duff}) in function of $c,$   for $v>0$: At each point
$(u,v)\in \mathcal{W}$ with $v>0,$ the vector $X_c(u,v)$ rotates at a non-zero speed in the direct sense when $c$ increases.  An easy consequence is that each segment of orbit starting at a regular point is also rotating (see \cite[Lemma 3.6]{[BLR]} for a trivial proof and also a precise definition of rotated vector fields).

The evolution with respect to a parameter of a separatrix starting at a saddle point was also largely studied,
in particular because  this question is of great interest  for the bifurcation theory  of vector fields (see, e.g., \cite{[A],Soto,[M]}).

Consider, as above, a family of smooth vector fields  $X_c$  with a family $\gamma_c$ of orbits \textcolor{black}{corresponding}  to  stable \textcolor{black}{separatrix} at a fixed saddle point $p_0,$ with $c\sim \bar{c}$. We choose a  section $\sigma$  with coordinate $r,$  transverse to the orbits  $\gamma_c$,
and consider the smooth function $D(c)$ which is the $r$-coordinate on $\sigma$ for the intersection point of $\gamma_c$ with $\sigma$, with $\bar{r}=D(\bar{c})$.  A classical formula asserts that the derivative of $D(c)$  at $\bar{c}$
can be computed by means of the (convergent) improper integral, called the {\it Melnikov integral:}\textcolor{black}{
$$ \frac{dD}{dc}(\bar{c})=\frac{1}{{\rm Det}(\partial_r\sigma,X_{\bar{c}}(\sigma(\bar{r}))}
\int _0^{+\infty}e^{-\int_0^t{\rm Div}(X_{\bar{c}}(\gamma_{\bar{c}}(s))ds}{\rm Det}(X_{\bar{c}},\partial_c X_{\bar{c}})(\gamma_{\bar{c}}(t)) dt.$$}

We return now to system \eqref{eq:p12}. For $\bar{u} \in [\theta_{hl},\theta_{+})$, let $t(\bar{u})$ be the time such that $u_{\gamma}(t(\bar{u},c))=\bar{u}$.   We can apply the above formula to compute the partial derivative $\partial_c v(\bar{u},\bar{c})$.
\vspace*{-1pt}
\begin{proposition}\label{melnikov}
	For any $\bar{u}\in [\theta_{hl},\theta_+)$  and $\bar{c}\in \R^+$, we have that
	$\partial_c v(\bar{u},\bar{c})<0.$ More explicitly, we have that:
	\begin{equation}\label{eq1}
		\partial_cv(\bar{u},\bar{c})=-\frac{1}{v(\bar{u},\bar{c})}\int _0^{+\infty}e^{-\bar{c}t}v^2_\gamma(t+t(\bar{u}),\bar{c})dt,
	\end{equation}
where $\gamma_{\bar{c}}(t)=(u_\gamma(t,\bar{c}),v_\gamma(t,\bar{c}))$  is defined by \eqref{gamma_c}. By choice, we have that
	$\gamma_{\bar{c}}(t)\rightarrow p_0$ for $t\rightarrow +\infty.$
\end{proposition}
\begin{proof}
For system \eqref{eq:p12} and in the notations introduced above, we can choose, as section $\sigma,$ the axis $\{u=\bar{u}\},$ with coordinate $v>0$. We have to replace $\gamma_{\bar{c}}(t)$ with $\gamma_{\bar{c}}(t +t(\bar{u}))$, which is the trajectory with initial condition $\gamma_{\bar{c}}(t(\bar{u}))$ at $t=0$.
Then, $\partial_r\sigma$ is the vector $\partial_v$ and ${\rm Det}(\partial_r\sigma,X_{\bar{c}}(\sigma(\bar{r}))= -v(\bar{u},\bar{c}).$
Moreover,
\begin{align*}
{\rm Div}(X_{\bar{c}}(\gamma_{\bar{c}}(s+t(\bar{u})))=\textcolor{black}{\bar{c}}, \quad {\rm Det}(X_{\bar{c}},\partial_c X_{\bar{c}})(\gamma_{\bar{c}}(t+t(\bar{u})))=\textcolor{black}{v_\gamma^2(t+t(\bar{u}),\bar{c})}.
\end{align*}
This gives the formula (\ref{eq1}).  We have that $v(\bar{u},\bar{c})>0$ and $e^{-\bar{c}t}v^2_\gamma(t,\bar{c})>0.$ This implies that
$\partial_c v(\bar{u},\bar{c})<0.$
\end{proof}

\subsection{The function $\psi(R)$}\label{4.3}

For $R\in \R^+$, we  want to find a $c(R),$  a solution of the equation:
\begin{equation}\label{eq2}
	c\theta_{hl}-qR=v(\theta_{hl},c)
\end{equation}
In fact, we can easily solve (\ref{eq2}) in terms of \textcolor{black}{$R$:}
$$R= \frac{1}{q}(c\theta_{hl}-v(\theta_{hl},c))$$
The function $c\mapsto   \frac{1}{q}(c\theta_{hl}-v(\theta_{hl},c))$ is     defined for $c>0.$ By Proposition \ref{melnikov}, its derivative is positive for all $c>0.$ Then, this function has an inverse function $\psi(R)$, also with a positive derivative
$\psi'(R)>0$,  for all $R$ in the domain of definition of $\psi(R).$ \textcolor{black}{Clearly, the function $\psi(R)$ verifies the following proposition.
\begin{proposition}\label{theta-dirivatiove-R}
The solution $\theta_{R,c}(x)$ of \eqref{*}, found in \textcolor{black}{Proposition \ref{theta_R_c}}, verifies
\begin{align}
\frac{d}{dx}\theta_{R,c}(R)=c\theta_{hl}-qR
\end{align}
if and only if $c=\psi(R)$.
\end{proposition}
}

Let us look at this domain of definition \textcolor{black}{of $\psi(R)$}.  The equation:
$$c\theta_{hl}=v(\theta_{hl},c)$$
has a unique solution which is $\psi(0)$ (the left term increases from $0$ to $+\infty$, while the right term \textcolor{black}{decreases} from $v(\theta_{hl},0)>0$ to $0).$

We have that $\psi(0)>0$ and $R\mapsto \psi(R)$ is a smooth map sending diffeomorphically the interval $[0,+\infty)$ on the interval
$[\psi(0),+\infty)$.

\textcolor{black}{Let $$M=\sup\{F(u)\,|\,u\in [\theta_{hl},\theta_+]\}>0.$$ If $V>0$ and $c>\frac{M}{V}$, the vector field $X_c$ is transverse and upward along the segment $[\theta_{hl},\theta_+]\times \{V\}$. Then, by the same argument used in Section \ref{4.1}, we have that the orbit $\gamma_c$ is contained in the strip $[\theta_{hl},\theta_+]\times (0,V)$. This means that, if $c>\frac{M}{V}$, then $$0<v(u,c)<V, \quad u\in [\theta_{hl},\theta_+).$$ As a consequence, $v(\theta_{hl},c)\to 0$ as $c\to +\infty$, and the graph of $\psi(R)$ is asymptotic to the line \textcolor{black}{($\Delta$)} when $R\to +\infty$.}

Finally, as \textcolor{black}{$$v(\theta_{hl},c(R))\leqslant v(\theta_{hl},0),$$}
we have that $$\psi(R)\leqslant \frac{1}{\theta_{hl}}(\textcolor{black}{v(\theta_{hl},0)}+qR).$$
Now, recalling  that $$\textcolor{black}{v(\theta_{hl},0)}=\Big(2\int_{\theta_{hl}}^{\textcolor{black}{\theta_+}}F(u)du\Big)^{\frac{1}{2}},$$ we obtain
$$\psi(R)\leqslant \frac{1}{\theta_{hl}}\Big(\big(2\int_{\theta_{hl}}^{\textcolor{black}{\theta_+}}F(u)du\big)^{\frac{1}{2}}+qR\Big)=c_+(R),$$ which completes the proof of Theorem \ref{theorem3}.

\section{Proof of Theorem \ref{t:1}} \label{s:5}

Since Theorem \ref{theorem3} is applicable  to the special case \eqref{eq:p4}, 
the results of Theorem \ref{theorem2} \and Theorem \ref{theorem3} can be combined.

\begin{corollary}\label{corol}
	\textcolor{black}{Let $F$ be defined by \eqref{eq:p4}.} 
There exists a unique pair $(R^{\ast},c^{\ast})$ belonging to the interior of the domain $Q_{+}$ such that
	\begin{equation}
		c^{\ast}= \varphi(R^{\ast})=\psi(R^{\ast}).
	\end{equation}
\end{corollary}
\begin{figure}[h]
	\centering \includegraphics[width=3in]{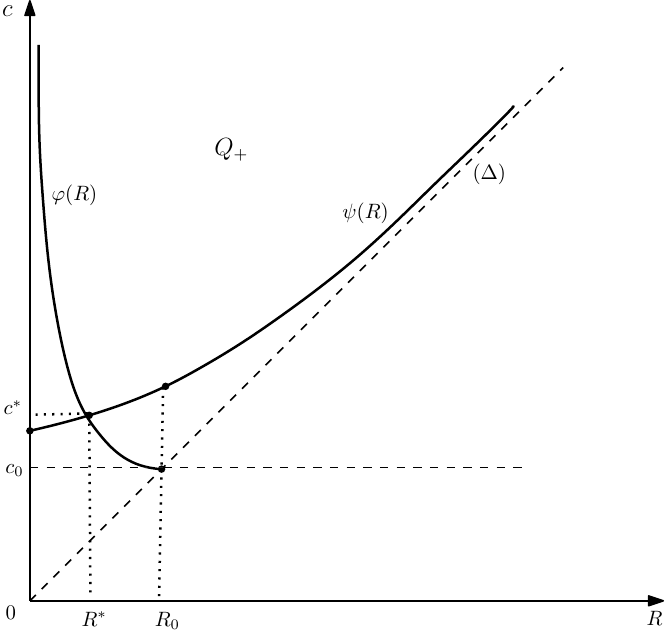}
	\caption{
 \textcolor{black}{The point $(R^{\ast},c^{\ast}) \in \interior({Q}_{+})$} is the only intersection between the curves $\varphi(R)$ and $\psi(R)$.
	Note that the representation of $\psi$ is merely illustrative.}
\label{Figgraph0c1}
\end{figure}

\begin{proof} The corollary easily follows from  Theorem \ref{theorem2} and Theorem \ref{theorem3}, since the smooth decreasing curve $R\mapsto \varphi(R)$ and the smooth increasing curve  \textcolor{black}{$R\mapsto \psi(R)$} have  a unique intersection that belongs to the interior of $Q_{+}$.

More precisely, we have that $$\psi(0)<\lim\limits_{R\to 0^+}\varphi(R)=+\infty$$ and
\begin{equation} \varphi(R_0)=\frac{q}{\theta_{hl}}R_0<\psi(R_0)=\frac{q}{\theta_{hl}}R_0+\textcolor{black}{\frac{v(\theta_{hl},\psi(R_0))}{\theta_{hl}}},
\end{equation}
as $v(\textcolor{black}{\theta_{hl}},\psi(R_0))>0$.
Then, there exists a $(R^{\ast},c^{\ast})$ \textcolor{black}{(see Figure \ref{Figgraph0c1})} such that $$c^{\ast}=\varphi(R^{\ast})=\psi(R^{\ast}),$$ by the \textcolor{black}{intermediate value theorem} and this value is unique, as the function $R \mapsto (\psi-\varphi)(R)$ is increasing.
\end{proof}

We are now in position to prove Theorem \ref{t:1} which is the main result of this paper.
\begin{proof}[Proof of Theorem \ref{t:1}]
(i) Let $(R^{\ast},c^{\ast})$ be given by Corollary \ref{corol}. As in Section \ref{4.1}, we define the function $\theta^{\ast}(x)$ on the interval $\left[R^*,+\infty\right)$ as
\begin{equation}  \theta^{\ast}(x) =\theta_{R^{\ast},c^{\ast}}(x)=u_{\gamma}(x-R^{\ast}, c^{\ast}).
\end{equation}
According to Proposition \ref{theta_R_c}, $\theta^{\ast}$ is the unique solution of \textcolor{black}{\eqref{*}},
\begin{align*}
\theta^{\ast}(R^*)=\theta_{hl}\quad \text{and} \quad \theta^{\ast}(x)\to \theta_{+} \,\,\,\text{as} \,\,\,x\to +\infty.
\end{align*}
Moreover, $\theta^*$ is increasing on $\left[R^*,+\infty\right)$, and \textcolor{black}{according to Proposition \ref{theta-dirivatiove-R}},
$$\theta^{\ast}_{x}(R^*)=c^{\ast}\theta_{hl}-qR^{\ast}>0.$$
\textcolor{black}{In addition, we prove in the Appendix (see Proposition \ref{proposition A.1}) that there exists $K_0>0$, such that
	\begin{align}\label{exponential}
		\theta_+-\theta^*(x)\leqslant K_0e^{\lambda_-x},\:\:x\to+\infty.
	\end{align}}
Therefore, we have established that the triplet $(R^{\ast}, c^{\ast}, \theta^{\ast}) \in  (\R^{+})^2 \times C^1([R^*,+\infty))$ is the unique solution of system \eqref{eq:p10}.\par

(ii) Finally, we return to the system \eqref{eq:p1}-\eqref{eq:p6}. We extend the function $\theta^{\ast}$ to the interval $\left(-\infty,R^*\right)$ as
\begin{eqnarray}
	\theta^{*}(x)=
	\left\{
	\begin{array}{ll}
		\theta_{ig} e^{c^{*}x}, & x\in (-\infty,0), \\[1mm]
		\theta_{ig} e^{c^{*}x} +\frac{q}{c^{*}} x +\frac{q}{(c^*)^2} (1-e^{c^{*}x}), & x\in(0,R^{*}).
	\end{array}
	\right.
\end{eqnarray}
It is clear that $\theta^{*}$ is of class $C^1$ on $(-\infty,R^{*})$, and
\begin{align*}
\theta^{*}(x)\to 0\,\,\,\text{ as}\,\,\, x\to -\infty.
\end{align*}
 Since $c^*=\varphi(R^*)$ (see Corollary \ref{corol}), the pair $(R^{\ast},c^{\ast})$ verifies the transcendental equation \eqref{eq:p81}, and it follows from \eqref{eq:p81}-\eqref{eq:p9} that both $\theta^{*}$ and its derivative $\theta^{*}_x$ have no-jump at $x=R^*$.
Together with part (i), we obtain that $\theta^{\ast}$ belongs to $C^1(\R)$ and the triplet $(R^{*},c^{*},\theta^{*})\in (\R^{+})^2 \times C^1(\R)$ uniquely solves \eqref{eq:p1}-\eqref{eq:p6}, up to translation. It is obvious  that $\theta^*$ is increasing on $(-\infty,0)$. For $0<x<R^*$, from \eqref{eq:p81}, we obtain
$$\theta^*_{x}(x)=(c^*\theta_{hl}-qR^*)e^{c^*(x-R^*)}+\frac{q}{c^*}(1-e^{c^*(x-R^*)}).$$
Since $c^*\theta_{hl}-qR^*>0$ and $1-e^{c^*(x-R^*)}>0$ for $0<x<R^*$,   $\theta^*(x)$ is also increasing on $(0,R^*)$. Therefore, $\theta^*$ is increasing on $\mathbb{R}$.
 This completes the proof of Theorem \ref{t:1}.
\end{proof}

\section*{Acknowledgments}
\textcolor{black}{C.-M.B. would like to express his sincere gratitude to the School of Mathematical Sciences at Tongji University for their warm hospitality during his visits in 2025 and 2026 as a Foreign Expert (National Program)}. M.M. and P.S. were supported by the National Natural Science Foundation of China (No. 12171368). The work of P.V.G was supported, in a part, by US-Israel BSF grant 2024033. The work of R.R. has been achieved in the frame of the EIPHI Graduate school (contract ANR-17-EURE-0002). \textcolor{black}{The authors are grateful to the anonymous reviewers for their detailed and relevant comments which lead to several substantial improvements  of the manuscript.}
\bigskip
\appendix
\section*{Appendix. More properties for $\theta_{R,c}(x)$}
\renewcommand{\thesection}{}
\renewcommand{\thefigure}{A.\arabic{figure}}
\renewcommand{\thesubsection}{A.\arabic{subsection}}
\renewcommand{\thetheorem}{A.\arabic{theorem}}
\renewcommand{\theequation}{A.\arabic{equation}}
\renewcommand{\theproposition}{A.\arabic{proposition}}
\renewcommand{\theremark}{A.\arabic{remark}}
\setcounter{figure}{0}
\setcounter{equation}{0}
\subsection{Introduction}\label{A.1}\mbox{} \vspace{0.5\baselineskip}

In Section \ref{4.1}, we have defined the function $\theta_{R,c}(x)=u_{\gamma}(x-R,c)$, for $x\in [R,+\infty)$. Its principal properties, proved in Proposition \ref{theta_R_c}, are that $\theta_{R,c}(x)\to \theta_+$ for $x\to +\infty$ and that $\frac{d\theta_{R,c}}{dx}>0$ for \textcolor{black}{$x\in [R,+\infty)$}. This function $\theta_{R,c}$ is defined for any function $F$ verifying the condition given in Section \ref{s:4} and for any $R>0$, $c>0$. 
\textcolor{black}{The intent of $\theta_{R,c}$ is to} construct the solution $\theta^*(x)=\theta_{R^*,c^*}(x)$ on the interval $[R,+\infty)$, as proved in Theorem \ref{t:1}.\par
In this appendix, we want to give \textcolor{black}{two additional} properties for $\theta_{R,c}$. \textcolor{black}{First,} the convergence of $\theta_{R,c}(x)$ to $\theta_+$ is exponential, as a consequence of a linearization theorem. \textcolor{black}{Second, if $F'(u)<0$ for $u\in [\theta_{hl}, \theta_+]$, then $\theta_{R,c}$ is concave on $[R,+\infty)$, i.e., $\frac{d^2\theta_{R,c}}{dx^2}(x)<0$ for $x\in[R,+\infty)$; this follows by an argument similar to the one already used in Section \ref{s:4}, for the function $v_c(u)$.}
\subsection{Exponential convergence at $+\infty$}\label{A.2}
\subsubsection{Sternberg-Chen linearization theorem}
The first significant result for the linearization of a differentiable system, at a singular point, was obtained by H. Poincar\'{e} \cite{P}, for complex analytic vector fields. This result was extended by S. Sternberg for real $C^{\infty}$ differentiable systems. He proved in \cite{S} that a smooth vector field is $C^{\infty}$-linearizable around a hyperbolic singular point, in any dimension, if some conditions of non-resonance among the eigenvalues at the singular point are fulfilled. A proof and some related references can be found in \cite{R}.\par
\textcolor{black}{In one dimension, the resonance conditions are automatically satisfied whenever the singularity is hyperbolic, which simply reduces to the condition that the sole eigenvalue $\lambda$ is nonzero. To circumvent technical subtleties in this setting, we invoke the one-dimensional linearization theorem of Sternberg and Chen (see \cite{Chen,S}), stated below.}

\begin{theorem}
	\label{exponal}
	Let $Z$ be a smooth real vector field defined on an interval $[u_0,u_1]$, $u_0\leqslant0\leqslant u_1$,
	\begin{align}
		Z:\frac{du }{dt}=f(u),\:\:\:0\leqslant t<+\infty.
	\end{align}
	Assume that $0$ is the unique root of $f$ on $[u_0,u_1]$. Moreover, we assume that $0$ is a hyperbolic singularity, namely $f'(0)=\lambda\neq 0$.\par
	Then, there exists a smooth diffeomorphism
	$$g:\:[x_0,x_1]\to [u_0,u_1],\:\:\:x_0\leqslant 0\leqslant x_1,$$
	with $g(0)=0$ and $g'(0)>0$, mapping the vector field $Z$ to the linear vector field $Z_0$
	\begin{align}
		Z_0:\:\frac{dx}{dt}=\lambda x,\:\:\:0\leqslant t<+\infty.
	\end{align}
	This means explicitly that:
	\begin{align}
		f(g(x))=\lambda x g'(x), \:\:\: x_0\leqslant x\leqslant x_1.
	\end{align}
	As a consequence, if $\lambda<0$ and $\varphi(t)$ is the trajectory of $u_1$ $(\text{i.e.,}\:\varphi(0)=u_1)$, then
	\begin{align}
		\varphi(t)=g(x_1e^{\lambda t}), \:\:\:0\leqslant t<+\infty.
	\end{align}
\end{theorem}
\textcolor{black}{\begin{remark}
		The assumption $g'(0) > 0$ in Theorem \ref{exponal} can be guaranteed by re-scaling the variable. More precisely, suppose that there exists a smooth diffeomorphism $g_0$ satisfying $g_0(0) = 0$, $g_0'(0) = a < 0$, and
		\begin{equation}\label{sbst}
			f(g_0(x)) = \lambda x g_0'(x), \quad x_0 \leqslant x \leqslant x_1.
		\end{equation}
		Let us introduce a scaled variable $\xi = bx$ ($b < 0$) and define a new mapping $g(\xi) \triangleq g_0\big(\frac{\xi}{b}\big)$, then $g(0) = 0$, $g'(0) = \frac{a}{b} > 0$. Now substitute $x = \frac{\xi}{b}$ into \eqref{sbst}, we can obtain $f(g(\xi)) = \lambda \xi g'(\xi),\ b x_1 \leqslant \xi \leqslant b x_0$. Thus, we construct a smooth diffeomorphism $g: [b x_1, b x_0] \to [u_0, u_1]$, $b x_1 < 0 < b x_0$ mapping the vector field $Z$ to the linear vector field $Z_0$ with $g'(0) > 0$.
	\end{remark}
}
\subsubsection{Proof of the exponential convergence}
We recall that for fixed $c>0$, $S_c$ is the half part of the stable manifold located in the region $[\theta_{hl},\theta_+]\times \mathbb{R}^+$. It is the union of an orbit $\gamma_c$ of the vector field $X_c$ (see \eqref{notations}-\eqref{eq:p12}), $\gamma_c(t)\to p_0=(\theta_+,0)$, and the limit point $p_0$ (see Figure \ref{gamma c}).\par
\begin{figure}[h]
\centering \includegraphics[width=4in]{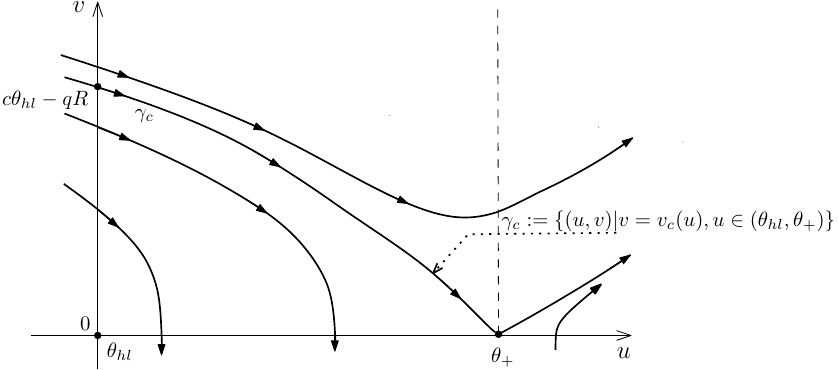}
\caption{Details of the phase portrait of $X_c$ (saddle point).}
\label{gamma c}
\end{figure}
We want to consider the one-dimensional vector field $Z$, restriction of $X_c$, along its invariant curve $S_c$. (In general, we will not mention $c$ in what follows). As $S_c$ is a graph over the interval $[\theta_{hl},\theta_+]$ of the smooth function $v_c(u)$, we can choose $u$ as variable, and $Z$ has a differentiable equation
\begin{equation}\label{acf}
\frac{du}{dt}=f(u),
\end{equation} for a smooth function $f(u)$, $u\in [\theta_{hl},\theta_+]$.\par
In fact, it is easy to check that $f(u)=v_c(u)$ \textcolor{black}{in our case}, but this information will not be used hereafter. \textcolor{black}{And, from Section \ref{s:2}, we can} verify that $\theta_+$ is the unique root of $f$ on $[\theta_{hl},\theta_+]$ and that $f'(\theta_+)=\lambda_-$, which is the negative eigenvalue of $X_c$ at the point $p_0=(\theta_+,0)$, defined in \eqref{eigenvalues}. We also have that $f(\theta_{hl})=c\theta_{hl}-qR$.\par
It is convenient to make the following change of variable:
\begin{align}\label{variable}
U=\theta_+-u.
\end{align}
\textcolor{black}{By abuse of language}, we go on to write
\begin{align}
Z:\:\frac{dU}{dt}=\tilde{f}(U),
\end{align}
where $\tilde{f}(U)=-f(\theta_+-U)$ and  $\tilde{f}:\:[0,\theta_+-\theta_{hl}]\to [qR-c\theta_{hl},0].$\par
On the other hand, it follows from Theorem \ref{exponal} that there exists a $C^\infty$-\break diffeomorphism
$$g:\:[0,x_1]\to[0,\theta_+-\theta_{hl}],$$
with
\begin{align}\label{g(x_0)}
g(x_1)=\theta_+-\theta_{hl},
\end{align}
which maps the vector field $Z_0$ to $X_c$,
\begin{align*}
Z_0:\:\frac{dx}{dt}=\tilde{f}'(0)x,\quad 0\leqslant t<+\infty,
\end{align*}
\textcolor{black}{with $\tilde f'(0)=\lambda_-$}, which means that
\begin{align*}
\tilde{f}(g(x))=\lambda_-xg'(x),\:\:\:0\leqslant x\leqslant x_1.
\end{align*}\par
It follows from Theorem \ref{exponal} that if $U(t)$ is the trajectory from $g(x_1)$ defined by \eqref{g(x_0)}, then
\begin{align}\label{U}
U(t)=g(x_1e^{\lambda_-t}),\:\:\: 0\leqslant t<+\infty.
\end{align}
We can deduce from $\eqref{U}$ the following estimate.
\begin{proposition}\label{proposition A.1}
There exists $K_0>0$  such that
\begin{align}\label{inequality_theta to theta+}
	\theta_+-\theta_{R,c}(x)\leqslant K_0e^{\lambda_-x},\:\:x\to+\infty.
\end{align}
\end{proposition}
\begin{proof}
Let $g$ be the diffeomorphism
\begin{align*}
	g:[0,x_1]\to[0,\theta_+-\theta_{hl}],
\end{align*}
with $g(0)=0$,
$g(x_1)=\theta_+-\theta_{hl}$ and
$g'(0)>0$.
As $$g(x)=g'(0)x+O(x^2),\:\: x\to 0,$$
we have that
\begin{align*}
	g(x_1e^{\lambda_-t})=g'(0)x_1e^{\lambda_-t}+O(e^{2\lambda_-t}),\:\: t\to+\infty.
\end{align*}
From \eqref{variable} and \eqref{U}, we have that
\begin{align}
	U(t)=\theta_+-u(t)=g'(0)x_1e^{\lambda_-t}+O(e^{2\lambda_-t}),\:\:t\to+\infty,
\end{align}
hence
\begin{align}\label{exponent}
	\theta_+-u(t)=O(e^{\lambda_-t}),\:\:t\to+\infty.
\end{align}

We return now to the function $\theta_{R,c}(x)$. Recall \textcolor{black}{from \eqref{notations} and Section \ref{s:4}} that $\theta_{R,c}(x)=u_{\gamma}(x-R,c)=u(x-R)$, where \textcolor{black}{the solution $u(t)$ to \eqref{acf} is also} the notation for $u_{\gamma}(t,c)$. As a consequence of \eqref{exponent}, we have that
$$\theta_+-\theta_{R,c}(x)=O(e^{\lambda_-(x-R)}),\:\:x\to+\infty,$$
and then
$$\theta_+-\theta_{R,c}(x)\leqslant K_0e^{\lambda_-x},\:\:x\to+\infty,$$
for some $K_0>0$.
\end{proof}
\subsection{Concavity of $\theta_{R,c}(x)$}\label{A.3}\mbox{} \vspace{0.5\baselineskip}

  Figure \ref{fig:2} suggests that \textcolor{black}{$\theta_{R,c}$} is concave, i.e., $\frac{d^2\textcolor{black}{\theta_{R,c}}}{dx^2}(x)<0$ for $x\in [R,+\infty)$, at least in the special case \eqref{Fu}, i.e. $F(u)=q-h((1+u)^4-1)$. \textcolor{black}{We aim to establish this concavity for all $\theta_{R,c}$, under mild assumptions on $F(u)$}. In fact, this second derivative is related to the first derivative of the function $v_c(u)$, as we explain now.\par
As $\theta_{R,c}(x)=u_{\gamma}(x-R,c)$, where $(u_{\gamma}(t,c),v_{\gamma}(t,c))$ are the components of the trajectory $\gamma_c(t)$ (see Section \ref{s:4}), then    \begin{align}
\frac{d\theta_{R,c}}{dx}=\frac{du_{\gamma}}{dt}(x-R,c)=v_{\gamma}(x-R,c).
\end{align}
By a new derivative, we obtain
\begin{align}
\frac{d^2\theta_{R,c}}{dx^2}=\frac{dv_c}{du}(u_{\gamma}(x-R,c))\frac{du_{\gamma}}{dt}(x-R,c).
\end{align}
As $\frac{du_{\gamma}}{dt}(x-R,c)>0$ for $x\in [R,+\infty)$, we have that $\frac{d^2\theta_{R,c}}{dx^2}$ and $\frac{dv_c}{du}$ have the same sign. We want to prove now the concavity of $\theta_{R,c}$ under
a mild condition on $F$.
\begin{proposition}
If $F(u)$ satisfies
\begin{equation}\label{conditionM}
	F'(u)<0,\quad u\in [\theta_{hl},\theta_+],
\end{equation}
we have that
\begin{align}
	v'_c(u)=\frac{dv_c}{du}(u)<0,\quad u\in [\theta_{hl},\theta_+].
\end{align}
As a consequence,
\begin{align}
	\frac{d^2\theta_{R,c}}{dx^2}<0,\quad x\in [R,+\infty).
\end{align}
\end{proposition}
\begin{proof}
Assuming that $F(u)$ satisfies \eqref{conditionM}, we consider the homoclinic curve
\begin{align}
	H_c=&\{cv-F(u)=0\,|\,u\in [\theta_{hl},\theta_+]\}.
\end{align}
The vector field $X_c$ is horizontal, rightway, and transverse to $H_c$, along $H_c\setminus\{(\theta_+,0)\}$.\par
We now consider the curved triangle $\mathcal{B}_1$, obtained \textcolor{black}{by replacing the upper side $S_0$ of the triangle $\mathcal{B}$ (defined in Section \ref{s:4})  by the homoclinic curve $H_c$}. From a direct computation, at the point $p_0$, the slope $\lambda_-$ of the stable eigenspace at the saddle, is greater than the slope $\frac{1}{c}F'(\theta_+)$ of $H_c$, i.e.,
\begin{align*}
	\frac{1}{c}F'(\theta_+)<\lambda_-<0.
\end{align*}
As a consequence, the trajectory $\gamma_c$ is contained in the interior of $\mathcal{B}_1$, for $t$ large enough. We can now repeat the argument given in Section \ref{s:4} with the triangle $\mathcal{B}$, to prove that the graph $S_c$ of $v_c(u)$ is contained in $\mathcal{B}_1$. More explicitly, we have that
\begin{align*}
	0<v_c(u)<\frac{F(u)}{c},\quad u\in [\theta_{hl},\theta_+).
\end{align*}
As the derivation of $v_c$ is given by
\begin{align}
	v_c'(u)=\frac{dv_c}{du}(u)=\frac{cv_c(u)-F(u)}{v_c(u)},
\end{align}
we have that $v_c'(u)<0$ for $u\in [\theta_{hl},\theta_+)$ (recall that, in particular $v_c'(\theta_+)=\lambda_-<0$).
\end{proof}
\begin{remark}
\textcolor{black}{According to \eqref{F'<0},}   condition \eqref{conditionM} is verified in the special case \eqref{Fu}.
\end{remark}










\end{document}